\documentclass[10pt,leqno]{amsart}
\usepackage{amsmath,amsthm,amsfonts,eucal,eufrak,enumerate}
\usepackage{latexsym}
\usepackage{amssymb}

\usepackage[dvips]{epsfig}
\usepackage[active]{srcltx}
\renewcommand{\thesection}{\arabic{section}}

\newtheorem{theorem}{Theorem}

\newtheorem{lemma}{Lemma}[section]
\newtheorem{prop}[lemma]{Proposition}
\newtheorem*{defi}{Definition}

\newtheorem*{remark}{Remark}

\renewcommand{\theequation}{\thesection .\arabic{equation}}
\let\subs\subsection
\renewcommand\subsection{\setcounter{equation}{0}
\gdef\theequation{\thesubsection \arabic{equation}}\subs}
\let\sect\section
\renewcommand\section{\setcounter{equation}{0}
\gdef\theequation{\thesection .\arabic{equation}}\sect}

\newcommand{\IR}{{\mathbb{R}}}

\newcommand{\IZ}{{\mathbb{Z}}}
\newcommand{\zv}{\IZ^\nu}

\newcommand{\be}{\begin{equation}}
\newcommand{\ee}{\end{equation}}

\newcommand{\AC}{{\rm{AC}}}
\newcommand{\loc}{{\rm{loc}}}
\newcommand{\ac}{{\rm{ac}}}
\newcommand{\sing}{{\rm{sing}}}

\newcommand{\supp}{\mathop{\rm{supp}}}

\newcommand{\dist}{\mathop{\rm{dist}}}

\newcommand{\Leb}{\mathop{\rm{Leb}}}

\newcommand{\ve}{\varepsilon}

\begin{document}

\smallskip

\title{Almost periodicity in time of solutions of the KdV equation}

\author{Ilia Binder}

\address{Department of Mathematics, University of Toronto, Bahen Centre, 40 St. George St., Toronto, Ontario, CANADA M5S 2E4}

\email{ilia@math.toronto.edu}

\thanks{I.\ B.\ was supported in part by an NSERC Discovery grant.}

\author{David Damanik}

\address{Department of Mathematics, Rice University, 6100 S. Main St. Houston TX 77005-1892, U.S.A.}

\email{damanik@rice.edu}

\thanks{D.\ D.\ was supported in part by NSF grant DMS--1361625.}

\author{Michael Goldstein}

\address{Department of Mathematics, University of Toronto, Bahen Centre, 40 St. George St., Toronto, Ontario, CANADA M5S 2E4}

\email{gold@math.toronto.edu}

\thanks{M.\ G.\ was supported in part by an NSERC Discovery grant.}

\author{Milivoje Lukic}

\address{Department of Mathematics, University of Toronto, Bahen Centre, 40 St. George St., Toronto, Ontario, CANADA M5S 2E4 and Department of Mathematics, Rice University, Houston TX 77005, U.S.A.}

\email{mlukic@math.toronto.edu}

\thanks{M.\ L.\ was supported in part by NSF grant DMS--1301582.}

\thanks{D.\ D., M.\ G., and M.\ L. would also like to thank the Isaac Newton Institute for Mathematical Sciences, Cambridge, for support and hospitality during the programme ``Periodic and Ergodic Spectral Problems" where part of this work was undertaken.}

\begin{abstract}
We study the Cauchy problem for the KdV equation $\partial_t u - 6 u \partial_x u + \partial_x^3 u = 0$ with almost periodic initial data $u(x,0)=V(x)$. We consider initial data $V$, for which the associated Schr\"odinger operator is absolutely continuous and has a spectrum that is not too thin in a sense we specify, and show the existence, uniqueness, and almost periodicity in time of solutions. This establishes a conjecture of Percy Deift for this class of initial data. The result is shown to apply to all small analytic quasiperiodic initial data with Diophantine frequency vector.
\end{abstract}

\maketitle

\tableofcontents

\section{Introduction}

This paper is devoted to the Cauchy problem for the KdV equation,
\begin{equation}\label{KdV}
\partial_t u - 6 u \partial_x u + \partial_x^3 u = 0,
\end{equation}
\begin{equation}\label{KdVinitial}
u(x,0) = V(x)
\end{equation}
with initial data $V:\mathbb{R} \to \mathbb{R}$. It is assumed that $u(x,t): \mathbb{R}^2 \to \mathbb{R}$ and solutions are considered in the classical sense: $u$ is three times differentiable in $x$, once differentiable in $t$, and obeys \eqref{KdV} for each $(x,t) \in \mathbb{R}^2$.

The KdV equation is named after Korteweg--de Vries~\cite{KdV}, who proposed it as a model for the propagation of shallow water waves in a canal. It has received a lot of attention after the discovery in 1960s of infinitely many conserved quantities and the inverse scattering transform by Gardner--Greene--Kruskal--Miura \cite{GGKM,GKM} and the Lax pair formalism \cite{Lax} which showed isospectrality of the associated Schr\"odinger operators $-\partial_x^2 + u(\cdot, t)$. Due to these properties, the KdV equation is often described as completely integrable. This formalism was first implemented for sufficiently fast decaying initial data $V$, where the conserved quantities are integrals of polynomial expressions in the solution and its derivatives, the lowest one being the (square of the) $L^2$ norm, $\int_{\mathbb{R}} \lvert u(x,t)\rvert^2 dx$.

In the 1970s, this framework was successfully extended to solutions of the KdV equation on a torus, $u:\mathbb{T} \times \mathbb{R} \to \mathbb{R}$, where the conserved quantities, of course, take the form of integrals on the torus, e.g. $\int_{\mathbb{T}} \lvert u(x,t)\rvert^2 dx$. It was shown, through the contributions of many authors  \cite{Ga,FZ,Du,No,McKvM,DMN,FMcL,McKT,Du2}, that in this setting, the KdV equation is a completely integrable Hamiltonian system with action-angle variables; for a book treatment, see \cite{KP}. In particular, the existence of action-angle variables implies that solutions are almost periodic in $t$. Of course, solutions on $\mathbb{T} \times \mathbb{R}$ can be viewed as spatially periodic solutions on $\mathbb{R} \times\mathbb{R}$, which motivates the following question.

Deift~\cite{De} posed an open problem whether, for \emph{almost periodic} initial data $V$, the solution of \eqref{KdV}, \eqref{KdVinitial} is almost periodic as a function of $t$. The analysis of the KdV Cauchy problem with almost periodic initial data presents significant new obstacles. The conserved quantities in this setting would take the form of spatial averages, e.g.
\[
\lim_{L\to\infty} \frac 1{2L} \int_{-L}^L \lvert u(x,t) \rvert^2 dx,
\]
and it is not known how to use such conserved quantities to prove even short time existence of a solution. Moreover, isospectrality of Schr\"odinger operators is too weak a condition to restrict the phase space of our solution to a torus, a first prerequisite for true integrability. In addition, unlike Schr\"odinger operators with fast-decaying or periodic potentials, almost periodic Schr\"odinger operators have diverse spectral properties.

In this paper, we answer Deift's question for the case of small quasiperiodic analytic initial data with Diophantine frequency. Let us first define the appropriate class of quasiperiodic functions and then state the theorem.

\begin{defi}
Let $\ve >0$, $0 < \kappa_0 \le 1$, and $\omega \in \mathbb{R}^\nu$ for some $\nu \in \mathbb{N}$. We say that $V \in \mathcal{P}(\omega,\ve,\kappa_0)$ if $V: \mathbb{R} \to \mathbb{R}$ is of the form
\begin{equation}\label{eq:samplingfunc}
V(x) = \sum_{n \in \zv} c(n) e^{2 \pi i n\omega x}\ ,
\end{equation}
such that
\begin{equation}\label{eq:samplingfunc2}
\lvert c(n)\rvert  \le  \ve\exp(-\kappa_0 \lvert n\rvert ), \quad \forall n \in \zv.
\end{equation}
We make $\mathcal{P}(\omega,\ve,\kappa_0)$ a metric space with the metric inherited from $L^\infty(\mathbb{R})$.
\end{defi}

\begin{theorem}\label{TQ}
Let $\omega \in \mathbb{R}^\nu$ obey the Diophantine condition
\begin{equation}\label{eq:dioph}
\lvert n \omega \rvert \ge a_0 \lvert n\rvert^{-b_0}, \quad n \in \zv \setminus \{ 0 \}
\end{equation}
for some
\begin{equation}\label{eq:dioph2}
0 < a_0 < 1,\quad \nu < b_0 < \infty.
\end{equation}
 There exists $\ve_0(a_0,b_0, \kappa_0)>0$ such that, if $\ve < \ve_0$ and  $V \in \mathcal{P}(\omega,\ve,\kappa_0)$, then:
 \begin{enumerate}[(i)]
\item there is a unique solution $u$ of \eqref{KdV}, \eqref{KdVinitial} such that
\begin{equation}\label{boundedness}
u, \partial_{xxx} u \in L^\infty(\mathbb{R} \times [-T, T]), \qquad  \forall T<\infty.
\end{equation}
\item for every $t \in \mathbb{R}$, $u(\cdot,t)$ is quasiperiodic; in fact, $u(\cdot, t) \in \mathcal{P}(\omega,\sqrt{4\ve},\kappa_0/4)$
\item $u$ is almost periodic in $t$, in the following sense: there is a compact (finite or infinite dimensional) torus $\mathbb{T}^d$, a continuous map
\[
\mathcal{M}: \mathbb{T}^{d} \to \mathcal{P}(\omega,\sqrt{4\ve},\kappa_0/4),
\]
a base point $\alpha \in \mathbb{T}^d$, and a direction vector $\zeta \in \mathbb{R}^{d}$ such that $u(\cdot, t) = \mathcal{M}(\alpha+\zeta t)$.
\end{enumerate}
\end{theorem}

Throughout this text, $\mathbb{T} = \mathbb{R} / 2\pi \mathbb{Z}$.

In \eqref{eq:samplingfunc2}, \eqref{eq:dioph}, and elsewhere in this work,  $\lvert n \rvert :=  \sum_j \lvert n_j \rvert$ for $n = (n_1, \dots, n_\nu) \in \zv$.

\begin{remark}
On $\mathcal{P}(\omega,\ve,\kappa_0)$, the $L^\infty$-norm is equivalent with the norm
\[
\lVert V - \tilde V \rVert_r = \left( \sum_{n\in \zv} \lvert c(n) - \tilde c(n) \rvert^2 e^{2\lvert n\rvert r} \right)^{1/2}
\]
for any $r < \kappa_0$, and with the Sobolev norm inherited from $W^{k,\infty}(\mathbb{R})$ for any $k\in \mathbb{N}$.  Part (iii) of the previous theorem therefore implies that besides $u$, derivatives of $u$ are also almost periodic in $t$, and so is each Fourier coefficient $c(n,t)$ of $u(x,t)$.
\end{remark}

The initial value problem \eqref{KdV}, \eqref{KdVinitial} for quasiperiodic initial data was previously approached by Tsugawa~\cite{Ts} and Damanik--Goldstein~\cite{DG2}, building upon work of Kenig--Ponce--Vega~\cite{KPV} and Bourgain~\cite{Bo}. \cite{Ts} proved local existence and uniqueness for initial data of the form \eqref{eq:samplingfunc} with Diophantine $\omega$ and $c(n)$ decaying at a sufficiently fast polynomial rate, and \cite{DG2} proved local existence and uniqueness for $V \in \mathcal{P}(\omega,\epsilon,\kappa_0)$ for all $\omega$ and global existence and uniqueness for Diophantine $\omega$. However, in all those results, uniqueness is proved within a class of solutions such that $u(\cdot, t)$ is quasiperiodic for all $t$, with similar decay conditions on Fourier coefficients as those assumed for the initial data (polynomial in \cite{Ts} and exponential in \cite{DG2}). Our Theorem~\ref{TQ} improves on this by proving uniqueness within a larger class of solutions assuming a priori only a boundedness condition \eqref{boundedness} and not quasiperiodicity. Moreover, while the class of initial data in Theorem~\ref{TQ} is the same as that considered in \cite{DG2}, this paper provides an alternative proof of the results from \cite{DG2} for Diophantine $\omega$. The alternative proof avoids the detailed combinatorial analysis in \cite{DG2} and relies instead on the inverse spectral theory of Schr\"odinger operators and fundamental work of Rybkin~\cite{Ry}, coupled with \cite{DGL2}. Finally, Theorem~\ref{TQ} proves almost periodicity of the solution in $t$, which was beyond the reach of the approach in \cite{Ts,DG2}.

In fact, we are able to prove a more general result, and prove existence, uniqueness, and almost periodicity in $t$ whenever $V$ is almost periodic and the spectrum of the associated Schr\"odinger operator $- \partial_x^2 + V$ is absolutely continuous and not too thin, in a sense to be made rigorous below. Stating that result requires introducing additional prerequisites.

For $W:\mathbb{R} \to \mathbb{R}$, consider the Schr\"odinger operator
\begin{equation} \label{eq:SchrOp}
[H_W y](x) = - y''(x) +  W(x) y(x), \quad x \in \IR.
\end{equation}
In the cases of interest in this paper, $W$ will be a bounded function, so \eqref{eq:SchrOp} defines an unbounded self-adjoint operator on $L^2(\mathbb{R})$ with the domain
\[
D(H_W) = D(H_0) = \{ y: \mathbb{R} \to \mathbb{R} \mid y\in \AC_\loc(\mathbb{R}), y' \in \AC_\loc(\mathbb{R}), y'' \in L^2(\mathbb{R})\}.
\]
The spectrum $\sigma(H_W)$ is the set of $z\in \mathbb{C}$ for which $H_W - z$ does \emph{not} have a bounded inverse map $(H_W - z)^{-1}:L^2(\mathbb{R}) \to D(H_W)$. Since $H_W$ is self-adjoint, $\sigma(H_W) \subset \mathbb{R}$.

For $\psi \in L^2(\mathbb{R})$, the spectral measure $d\mu_\psi$ is the unique measure on $\mathbb{R}$ with the property that
\[
\langle \psi, (H_W - z)^{-1} \psi \rangle = \int \frac 1{x-z} d\mu_\psi(x), \qquad \forall z \in \mathbb{C} \setminus \mathbb{R}.
\]
Using the Lebesgue decomposition of $d\mu_\psi = d\mu_{\psi,\ac} + d\mu_{\psi,\sing}$ with respect to Lebesgue measure, the absolutely continuous spectrum of $H_W$ can be defined as the smallest common topological support of absolutely continuous parts of all spectral measures,
\[
\sigma_\ac(H_W) = \overline{\bigcup_{\psi \in L^2(\mathbb{R})} \supp d\mu_{\psi,\ac}}.
\]
Clearly, $\supp d\mu_{\psi,\ac} \subset \supp d\mu_\psi \subset \sigma(H_W)$, so $\sigma_\ac(H_W) \subset \sigma(H_W)$. We will later be interested in cases in which $\sigma_\ac(H_W) = \sigma(H_W)$.

The spectrum $S = \sigma(H_W)$ is closed  and bounded from below but not from above, so it can be written in the form
\[
S = [\underline E, \infty) \setminus \bigcup_{j\in J} (E_j^-,E_j^+),
\]
where $\underline E = \inf S$ and $(E_j^-, E_j^+)$ are the maximal open intervals in $\mathbb{R} \setminus S$, called gaps. We will sometimes also work with the closure of $S$ in the Riemann sphere, denoted $\bar S = S \cup \{ \infty\}$.

Denote by $g_{\mathbb{C} \setminus S}(z)$ the potential theoretic Green's function for the domain $\mathbb{C} \setminus S$ with a logarithmic pole at $\underline E - 1$. The set $\bar S$ is regular if
\begin{equation}\label{regularity}
\lim_{\substack{z\in \mathbb{C} \setminus S \\ z \to x}} g_{\mathbb{C} \setminus S}(z) = 0, \quad \forall x\in \bar S.
\end{equation}
Denote by $c_j \in (E_j^-, E_j^+)$ the critical points of $g_{\mathbb{C} \setminus S}$ in the gaps. The Parreau--Widom condition states that
\begin{equation}\label{ParreauWidom}
\sum_{j\in J} g_{\mathbb{C} \setminus S}(c_j) < \infty.
\end{equation}

Regularity and the Parreau--Widom condition were first used in spectral theory in groundbreaking work of Sodin--Yuditskii \cite{SY,SY1}. The results in those papers are stated in terms of homogeneity in the sense of Carleson, that is, the existence of $\tau>0$ such that
\begin{equation}\label{Carleson}
\lvert S \cap [x_0 -  \epsilon, x_0+\epsilon] \rvert \ge \tau \epsilon, \quad \forall x_0 \in S, \quad \forall \epsilon \in (0,1],
\end{equation}
and finite gap length. However, it is then remarked that those conditions imply regularity and the Parreau--Widom condition, and the remainders of their proofs rely only on that.

We will need additional conditions on the spectrum. We denote
\[
\gamma_j = E_j^+ - E_j^-
\]
for $j\in J$ and
\begin{align*}
\eta_{j,l} & = \dist((E_j^-,E_j^+), (E_l^-,E_l^+)) \\
\eta_{j,0} & = \dist( (E_j^-,E_j^+), \underline E)
\end{align*}
for $j, l\in J$ (notationally, we assume here that our abstract index set $J$ does not contain $0$ as an element). We also denote
\begin{equation}\label{Cjdef}
C_j = (\eta_{j,0} + \gamma_j)^{1/2} \prod_{\substack{l \in J \\ l\neq j}} \left( 1 + \frac{\gamma_l}{\eta_{j,l}} \right)^{1/2}.
\end{equation}
We assume that $S$ satisfies a set of conditions of Craig~\cite{Cr}:
\begin{equation}\label{Craig1}
\sum_{j\in J} \gamma_j < \infty, \qquad \sup_{j \in J} \gamma_j C_j < \infty,  \qquad  \sup_{j\in J} \frac{\gamma_j^{1/2}}{\eta_{j,0}} C_j < \infty, \qquad \sup_{j\in J} \sum_{\substack{l \in J \\ l\neq j}} \frac{\gamma_j^{1/2} \gamma_l^{1/2}}{\eta_{j,l}} C_j < \infty.
\end{equation}
Craig introduced these conditions to ensure that a vector field, associated with the translation flow, is Lipshitz on the isospectral torus $\mathcal{D}(S) = \mathbb{T}^J$ with the metric given by
\begin{equation}\label{metricCraig}
\lVert \varphi - \tilde\varphi \rVert_{\mathcal{D}(S)} = \sup_{j\in J} \gamma_j^{1/2} \lVert \varphi_j - \tilde \varphi_j \rVert_{\mathbb{T}}.
\end{equation}

In order to ensure that a second vector field, associated with the KdV flow, is Lipshitz, we need to strengthen these conditions as follows:
\begin{equation}\label{Craig2}
\sum_{j\in J} \gamma_j^{1/2} < \infty,
\qquad \sup_{j \in J} \gamma_j^{1/2} \frac{1+\eta_{j,0}}{\eta_{j,0}} C_j < \infty,
\qquad \sup_{j\in J} \sum_{\substack{l \in J  \\ l \neq j}} \left( \frac{\gamma_j^{1/2} \gamma_l^{1/2}}{\eta_{j,l}} \right)^a (1 + \eta_{j,0} ) C_j < \infty \text{ for }a \in \left\{ \frac 12, 1 \right\}.
\end{equation}

Note that while these Craig-type conditions look complicated, they are quite weak. For instance, it is straightforward to see that they are satisfied for the small quasiperiodic initial data of Theorem~\ref{TQ}. Also note that \eqref{Craig2} implies \eqref{Craig1}.

We also wish to emphasize that this form of Craig-type conditions is sufficient, but not necessary. In particular, if one changed the metric \eqref{metricCraig}, the form of the conditions \eqref{Craig1}, \eqref{Craig2} would change. At this stage in the theory, the fine structure of gap sizes and distances is not well understood for general almost periodic operators, so it is not clear that varying the metric would broaden the scope of the following theorem. We therefore choose, for clarity of exposition, to keep the metric \eqref{metricCraig} used by Craig.

To ensure continuity of scalar fields associated with certain trace formulas, we will also assume that
\begin{equation}\label{traceconditions}
\sum_{j\in J} (1+\eta_{j,0}^2) \gamma_j < \infty.
\end{equation}

\begin{theorem}\label{Tap}
Let the initial data $V:\mathbb{R} \to \mathbb{R}$ be uniformly almost periodic. Denote $S=\sigma(H_V)$ and assume that
\begin{enumerate}[(i)]
\item $S = \sigma_\ac(H_{V})$;
\item $S$ is a regular Parreau--Widom set;
\item $S$ obeys the Craig-type conditions \eqref{Craig2}, \eqref{traceconditions};
\end{enumerate}
Then:
\begin{enumerate}[(i)]
\item there is a unique solution $u$ of \eqref{KdV}, \eqref{KdVinitial} which obeys the boundedness condition \eqref{boundedness};
\item the solution $u$ is almost periodic in $t$, in the following sense: there is a continuous map
\[
\mathcal{M}: \mathbb{T}^{J} \to W^{4,\infty}(\mathbb{R}),
\]
a base point $\alpha \in \mathbb{T}^J$, and a direction vector $\zeta \in \mathbb{R}^{J}$ such that $u(\cdot, t) = \mathcal{M}(\alpha+\zeta t)$;
 \item for each $t\in\mathbb{R}$, the function $u(\cdot, t)$ is uniformly almost periodic with frequency module equal to the frequency module of $V$.
 \end{enumerate}
\end{theorem}

We recall that the frequency module of an almost periodic function $V$ is the $\mathbb{Z}$-module generated by the set of all $\lambda\in\mathbb{R}$ such that
\[
\lim_{L\to\infty} \frac 1{2L} \int_{-L}^L V(x) e^{-i\lambda x} dx \neq 0.
\]

The condition \eqref{boundedness} in Theorems~\ref{TQ} and \ref{Tap} is assumed only in order to use the results of Rybkin~\cite{Ry}; as discussed in Remark 2 on page 9 of \cite{Ry}, this condition can be relaxed.

A construction of Cohen--Kappeler~\cite{CK} shows nonuniqueness of the Cauchy problem \eqref{KdV},\eqref{KdVinitial} even for $V=0$ on some domains of the form $\{(x,t)\in\mathbb{R}^2 \vert 0 \le t < h(x)\}$ with $h$ a strictly positive, increasing function of $x$.

An affirmative answer to Deift's problem was already known in some cases. A class of ``algebro--geometric" solutions of the KdV equation was constructed in the 1970s, resulting from the same theory which solved the periodic case. These solutions correspond to almost periodic Schr\"odinger operators with finitely many gaps in the spectrum, so they are often called finite zone or finite gap solutions. Due to having finitely many gaps, their frequency module under any linear flow is finitely generated, so they are quasiperiodic in both $x$ and $t$. We emphasize, however, the difference between these finite gap solutions and our quasiperiodic initial data of Theorem~\ref{TQ}, which generically have infinitely many gaps in the spectrum.

A result of Egorova \cite{Eg} goes beyond the finite gap case to construct almost periodic solutions with gap sizes obeying a superexponential decay condition in a parameter resembling the period of periodic approximants. This result applies to a class of limit-periodic potentials previously studied by Chulaevsky~\cite{Ch} and Pastur--Tkachenko~\cite{PT1,PT2}, which can be approximated superexponentially in the period by periodic potentials. Namely, $V$ is assumed to have $a_n$-periodic potentials $V_n$ such that for any $b> 0$,
\[
\lim_{n\to \infty} e^{b a_{n+1}} \sup_{x\in\mathbb{R}} \lVert V - V_n \rVert_{W^{4,2}((x,x+1))} = 0.
\]
Due to the superexponential rate of approximation, this analysis can obtain convergence from various estimates exponential in the period. This approach would not apply in our setting.

While Theorem~\ref{TQ} is direct in the sense that it derives the desired conclusion from explicit assumptions about the initial datum, there is a sophisticated inverse spectral theory working in the background which makes this result possible. The inverse spectral theory depends on fundamental developments by Craig~\cite{Cr}, Sodin--Yuditskii~\cite{SY}, Remling~\cite{R07}, centered on the notion of reflectionless Schr\"odinger operators. Another fundamental development is due to Rybkin~\cite{Ry}, who studied the time evolution of solutions of \eqref{KdV} in terms of the Weyl $M$-matrices of the associated family of Schr\"odinger operators; since the $M$-matrix determines a Schr\"odinger operator uniquely, Rybkin's approach provides much more information about the solution than what can be concluded merely from unitary equivalence of operators.

We wish to emphasize that Theorem~\ref{Tap} is not itself a small coupling result, even though its corollary, Theorem~\ref{TQ}, is. It is natural to ask to what extent these properties persist as the coupling constant $\ve$ is increased. In this regard, a particularly well understood example is the almost Mathieu operator, a discrete Schr\"odinger operator on $\ell^2(\mathbb{Z})$ given by $(H_\lambda u)_n = u_{n-1} + u_{n+1} + 2 \lambda \cos(2 \pi(n\omega + \theta)) u_n$, where $\lambda >0$, $\omega \in \mathbb{R} \setminus \mathbb{Q}$, $\theta \in \mathbb{R}$. The almost Mathieu operator goes through a transition from absolutely continuous to singular spectrum as the coupling constant $\lambda$ goes through $1$. Moreover, the Lebesgue measure of the spectrum is $4\left\lvert  \lambda - 1 \right\rvert$, so $\lambda=1$ is also a critical point in that sense. Many authors have contributed to these results; see the recent survey papers \cite{Da, JiMa} and references therein. Recently, it was shown in \cite{DGSV} that the spectrum of the almost Mathieu operator is homogeneous whenever $\lambda \neq 1$ (under a Diophantine condition for $\omega$). This is encouraging evidence that for analytic quasiperiodic operators, the assumptions of Theorem~\ref{Tap} may hold as we increase the coupling constant, up to a critical value where there is a transition in the spectral type. However, this is currently in the domain of speculation.

More specifically, our work serves as motivation for one to work out a continuum analogue of Avila's global theory for discrete one-frequency analytic quasi-periodic Schr\"odinger operators \cite{Av1, Av2, Av3}, especially of his results in the subcritical regime. That is, one should classify cocycle behavior in the continuum case in the same way as Avila does in the discrete case, namely as subcritical/critical/supercritical, and prove that subcritical cocycle behavior leads to almost reducibility and absolutely continuous spectrum. The ultimate goal would be to show that our results apply throughout the subcritical regime. Of course, one would also have to establish the necessary structural properties for the spectrum as a set, and here the result from \cite{DGSV} suggests that these properties may indeed be typical in the subcritical regime.

While the previous paragraph describes ways in which the known results in the discrete case may inform a study of the continuum case, one can also turn this around and attempt to advance the theory in the discrete case by implementing ideas that have proved to be useful in the continuum case. For example, given the results obtained in the present paper, the following natural question arises: considering the Toda lattice,
\begin{equation}\label{toda}
\begin{split}
\dot{a}(n,t) &=a(n,t)(b(n+1,t)-b(n,t)),\\
\dot{b}(n,t) & =2(a^2(n,t)-a^2(n-1,t)),
\end{split}
\end{equation}
so that the initial data $(a(\cdot,0), b(\cdot,0))$ are almost periodic and the associated Jacobi matrix is reflectionless and its spectrum is a regular Parreau--Widom set and satisfies suitable Craig-type conditions, is it then true that there is a global solution that is almost periodic in time? Based on the close analogy between the KdV equation and the Toda lattice, as well as our Theorem~\ref{Tap}, we expect an affirmative answer to this question. Many of the key ingredients are already in place; for example, Remling has shown the reflectionlessness of almost periodic Jacobi matrices with fully supported absolutely continuous spectrum \cite{R11} and the invariance of reflectionlessness under the Toda flow \cite{R15}, and Sodin and Yuditskii's work exists in the discrete setting as well \cite{SY1}.

An attractive special case is of course given by almost Mathieu initial data,
$$
a(n,0) = 1, \quad b(n, 0) = 2 \lambda \cos(2 \pi (n\omega + \theta)).
$$
This special case is particularly interesting because, as indicated above, the spectral analysis of the initial data is almost complete. In particular, the phase transition at $\lambda = 1$ is understood on a variety of levels: Lebesgue measure of the spectrum, quantum transport, spectral type, behavior of (generalized) eigenfunctions, and more concretely the behavior of the underlying cocycles. On the one hand, this suggests that one may hope that by working out the discrete analogue of our work, one can establish the desired result on almost periodicity in time throughout the subcritical region (i.e., for $\lambda < 1$). On the other hand, this case exemplifies the need for significant additional ideas to penetrate the supercritical region ($\lambda > 1$). Of course, it is known that in the supercritical region, the initial data will not be reflectionless. However, the structure of the spectrum as a set is the same as in the subcritical region; indeed, the spectra corresponding to coupling constants $\lambda$ and $\lambda^{-1}$ are related via linear scaling due to Aubry duality. This shows clearly that isospectrality alone is insufficient to get a good understanding of the evolution (even with additionally imposed Craig-type estimates). At the present time, it is unclear how to define the appropriate ``isospectral torus'' in the supercritical regime, or more generally when reflectionlessness fails. The critical regime (in the sense of Avila, corresponding to the case of critical coupling, $\lambda = 1$, in the case of the almost Mathieu operator) poses additional challenges due to the critical nature of the associated cocycles and the degenerate structure of the spectrum.

While the critical regime is atypical in the class of discrete Schr\"odinger operators with analytic one-frequency quasi-periodic potentials, the features of the critical regime, such as spectra of zero Lebesgue measure and purely singular continuous spectral measures, are actually generic in the continuous category (i.e., varying the sampling function in $C^0$ while keeping the base transformation fixed). For example, it was shown by Avila \cite{Av} in the discrete case and then, based on his ideas, by Damanik--Fillman--Lukic \cite{DFL} in the continuum case that the spectrum of a generic bounded continuous limit-periodic operator has Lebesgue measure $0$. A result of Avila--Damanik~\cite{AD} shows that for discrete quasiperiodic Schr\"odinger operators, singular spectrum (and hence the absence of reflectionlessness!) is generic, and this suggests that the same is true for continuum quasiperiodic Schr\"odinger operators.\footnote{We expect that it is not too hard to work out a continuum analog of this particular result from \cite{AD}. Given the obvious relevance to the scope of our work in the context of the KdV equation, it would be nice to see this extension be worked out explicitly.} (It is conjectured that zero-measure spectrum is generic in this setting as well; however, this is not yet known, see \cite{ADZ} for a partial result.) Finally, for discrete Schr\"odinger operators, the absence of point spectrum was also shown to be generic in the continuous category by Boshernitzan-Damanik \cite{BD}. These results suggest that a complete answer to Deift's problem will require the development of robust new tools in inverse spectral theory.

In any event, the spectral analysis of almost periodic Schr\"odinger operators and Jacobi matrices is in general very complicated and just having a Lax pair is not enough to integrate the evolution equation (i.e., the KdV equation and the Toda lattice, respectively)

As a final remark in this introduction, let us point out that we obtain the almost periodicity of the solution $u$ as a function on $\mathbb{R}^2$. Namely, instead of writing $u(\cdot, t) = \mathcal{M}(\alpha+\zeta t)$ as in Theorem~\ref{Tap}, we can also write $u(x,t) = F(\alpha + \delta x + \zeta t)$ with a continuous function $F : \mathbb{T}^J \to \mathbb{R}$. This follows from our proof of the identity $u(\cdot, t) = \mathcal{M}(\alpha+\zeta t)$ (along with a suitable composition with the Dubrovin flow); compare Section~\ref{sec.5}.

Sections~\ref{sec.2} and \ref{sectionKdV} prepare the stage for the uniqueness part of Theorem~\ref{Tap}: Section~\ref{sec.2} starts with the necessary definitions and reviews and expands previous results on the evolution of Dirichlet data of a reflectionless operator under translation, and Section~\ref{sectionKdV} develops a similar description for the evolution of Dirichlet data under the KdV flow. Section~\ref{sec.4} is devoted to the existence part of Theorem~\ref{Tap}, and Section~\ref{sec.5} to the almost periodicity of the KdV flow. Section~\ref{sec.5} also contains the proof of Theorem~\ref{Tap}, based on all the previous material. Finally, Section~\ref{sec.6} uses Theorem~\ref{Tap} to prove Theorem~\ref{TQ}.

\section{The Translation Flow and Non-Pausing of Dirichlet Eigenvalues}\label{sec.2}

In this section, we discuss the translation flow applied to a reflectionless Schr\"odinger operator with spectrum $S$ obeying \eqref{Craig1}. In particular, this flow can be expressed in terms of the Dirichlet data of the operator. We begin by introducing the notion of reflectionlessness and reviewing results of Craig~\cite{Cr}, after which we investigate the behavior of Dirichlet data at gap edges to give a new characterization of points where a Dirichlet eigenvalue goes through a gap edge and, in particular, to explain why Dirichlet eigenvalues don't pause at gap edges.

For $z \in \mathbb{C} \setminus \sigma(H_W)$, the second order differential equation
\[
- y'' + W y = z y
\]
has nontrivial solutions $\psi_\pm(x;z)$, called Weyl solutions, such that $\psi_\pm(x;z) \in L^2([0,\pm\infty),dx)$. In terms of the Weyl solutions, one can express the half-line $m$-functions associated with the half-line restrictions of $H_W$ to $[x,\pm \infty)$ with a Dirichlet boundary condition at $x$ by
\[
m_\pm(x;z) = \frac{\psi'_\pm(x;z)}{\psi_\pm(x;z)}.
\]
For each $x$, these are meromorphic functions of $z\in \mathbb{C} \setminus \sigma(H_W)$.

The diagonal Green's function associated with $H_W$ is given by
\[
G(x,x;z) = \frac{\psi_-(x;z) \psi_+(x;z)}{\psi'_-(x;z) \psi_+(x;z) - \psi'_+(x;z) \psi_-(x;z)} = \frac 1{m_-(x;z) - m_+(x;z)},
\]
and it is an analytic function of $z\in \mathbb{C} \setminus \sigma(H_W)$ for each $x$.

The function $G(x,x;z)$ has nontangential limits $G(x,x;y+i0)$ for Lebesgue-a.e.\ $y\in \mathbb{R}$, and $H_W$ is called reflectionless if
\[
\Re G(x,x;y+i0) = 0 \quad \text{for Lebesgue-a.e.\ }y \in \sigma(H_W).
\]
It is known \cite{SY} that this definition is independent of the choice of $x$. We denote the set of reflectionless operators with spectrum $S$ by $\mathcal{R}(S)$.

Recall the set $\mathcal{D}(S) = \mathbb{T}^J$ with the metric \eqref{metricCraig}. We introduce variables on $\mathcal{D}(S)$ given by
\begin{align}
\mu_j & = E_j^{-} + (E_j^{+} - E_j^{-}) \cos^2 (\varphi_j / 2) \label{mufromvarphi} \\
\sigma_j & = \begin{cases} + & \varphi_j \in (0,\pi) + 2\pi \mathbb{Z} \\
- & \varphi_j \in (-\pi,0) + 2\pi \mathbb{Z} \\
0 & \varphi_j \in \pi \mathbb{Z}
\end{cases} \label{sigmafromvarphi}
\end{align}
On $\mathcal{D}(S)$, we introduce a scalar field $Q_1$,
\[
Q_1(\varphi) = \underline E + \sum_{j\in J} (E_j^- + E_j^+ - 2\mu_j)
\]
and a vector field $\Psi$, with components
\[
\Psi_j(\varphi) = 2 \left( (\mu_j - \underline{E})   \prod_{l\neq j} \frac{(E_l^- - \mu_j )(E_l^+ - \mu_j ) } {(\mu_l - \mu_j )^2} \right)^{1/2}.
\]
The tangent space of a point on $\mathcal{D}(S)$ will be equipped with the norm
\begin{equation}\label{tangentspacenorm}
\lVert v \rVert = \sup_{j\in J} \gamma_j^{1/2} \lvert v_j \rvert
\end{equation}
and vector fields on $\mathcal{D}(S)$ will be equipped with the $\sup$-norm obtained from \eqref{tangentspacenorm}.

The motivation for this setup comes from the Dirichlet data of a reflectionless operator on $S$.  Let $W$ be reflectionless and $\sigma(H_W) = S$. Following the definition of Craig \cite{Cr}, the Dirichlet data $\mu_j(x)$ are defined as zeros of the diagonal Green's function in the gaps. More precisely, using the observation that $G(x,x;z)$ is strictly increasing for $z\in (E_j^-, E_j^+)$, Craig defines
\begin{equation}\label{dirichleteigenvalueat0}
\mu_j(x) = \begin{cases}
z \in (E_j^-, E_j^+) & G(x,x;z) = 0 \\
E_j^+ & G(x,x;z) < 0\text{ for all }z \in (E_j^-, E_j^+) \\
E_j^- & G(x,x;z) > 0\text{ for all }z \in (E_j^-, E_j^+)
\end{cases}
\end{equation}
If $\mu_j(x) \in (E_j^-, E_j^+)$, the sign $\sigma_j(x) \in \{+, -\}$ is also defined. It is determined in terms of which of the half-line $m$-functions $m_+$, $m_-$ has a pole at $\mu_j$; namely, so that $m_{\sigma_j}$ has a pole at $\mu_j$. This definition immediately implies that for any choice of $z \in (E_j^-, E_j^+)$ and $\sigma \in \{+,-\}$,
\begin{equation}\label{zerosofweylsolution}
\{x \in \mathbb{R} \mid \mu_j(x) = z, \sigma_j(x) = \sigma\} = \{ x \in \mathbb{R} \mid \psi_\sigma(x;z) = 0\}
\end{equation}
where $\psi_\pm(x;z)$ are the Weyl solutions at the point $z \in (E_j^-, E_j^+)$.

One then also defines angular coordinates $\varphi_j(x) \in \mathbb{T}$ implicitly by \eqref{mufromvarphi}, \eqref{sigmafromvarphi}.

We therefore have, corresponding to the potential $W$, a trajectory $\varphi(x)$ on $\mathcal{D}(S)$. In particular, if we fix $x=0$, we obtain a map $B: \mathcal{R}(S) \to \mathcal{D}(S)$ given by
\begin{equation}\label{mapB}
B(W) = \varphi(0).
\end{equation}
Craig~\cite{Cr} showed the following:
\begin{enumerate}[(i)]
\item As a function of $x\in\mathbb{R}$, $\varphi(x)$ is a continuous trajectory on $\mathcal{D}(S)$.
\item The potential can be recovered from this trajectory by the trace formula
\[
W(x) = Q_1(\varphi(x)).
\]
\item The diagonal Green's function of $H_W$ is given by the product formula
\begin{equation}\label{greensfunctionproductformula}
G(x,x;z) = \frac 12 \sqrt{ \frac 1{\underline E - z} \prod_{l\in J} \frac{(\mu_l(x) -z)^2 }{(E_l^- - z)(E_l^+ - z)}  }
\end{equation}
with the branch of square root chosen so that $G(x,x;z)$ is analytic in $z\in \mathbb{C} \setminus S$ and $G(x,x;z)>0$ for $z< \underline E$. In particular,
\begin{equation}\label{greensfunctionderivativeatmu}
\left. \partial_z G(x,x;z) \right\rvert_{z=\mu_j(x)} = \frac 12 \sqrt{ \frac 1{(\underline E - \mu_j)(E_j^- - \mu_j)(E_j^+ - \mu_j)} \prod_{l\neq j} \frac{(\mu_l -\mu_j)^2 }{(E_l^- - \mu_j)(E_l^+ - \mu_j)}  }.
\end{equation}
\item $\Psi$ is a Lipshitz vector field on $\mathcal{D}(S)$.
\item Fix $j\in J$. At any point $x\in \mathbb{R}$ at which $\mu_j(x) \in (E_j^-, E_j^+)$, the function $\varphi_j(x)$ is differentiable and
\begin{equation}\label{translationflowj}
\partial_x \varphi_j(x) = \Psi_j(\varphi(x)).
\end{equation}
\item If $\mathcal{R}(S)$ is equipped with the topology of uniform convergence on compacts, the map $B$ is continuous.
\end{enumerate}

We wish to clarify the situation at the gap edges, with respect to (v) above. We will show that the characterization \eqref{zerosofweylsolution} has an analog for gap edges, although there are no Weyl solutions at a gap edge. In particular, this will show that $\mu_j(x)$ does not pause at a gap edge and that $\varphi_j(x)$ is differentiable and \eqref{translationflowj} holds at all $x$.

\begin{prop}\label{nonpausingprop} Assume that $S$ obeys \eqref{Craig1} and $W \in \mathcal{R}(S)$. Fix a gap $(E_j^-,E_j^+)$ and let $E \in\{ E_j^-, E_j^+\}$. There is a nontrivial solution $\tilde \psi$ of $H_W \tilde \psi = E \tilde \psi$ which is the limit of Weyl solutions from the gap, in the following sense: there are normalizing constants $c_\pm(z) \in \mathbb{C} \setminus \{0\}$ for $z\in (E_j^-, E_j^+)$ such that
\[
\lim_{\substack{z \in (E_j^-,E_j^+) \\ z \to E}}
c_\pm(z) \psi_\pm(x;z) = \tilde\psi(x)
\]
uniformly on compacts in $x$. For this solution $\tilde\psi$, we have
\begin{equation}\label{dirichletdatanonpausing}
\{x \in \mathbb{R} \mid \mu_j(x) = E\} = \{ x\in \mathbb{R} \mid \tilde\psi(x) =0\}.
\end{equation}
In particular, this set is discrete.
\end{prop}

As pointed out in \cite{Cr}, knowing continuity of $\varphi_j(x)$ and knowing that \eqref{translationflowj} holds away from a discrete set in $x$ implies that \eqref{translationflowj} holds for all $x$. Thus, Dirichlet data of $W$ are trajectories of the flow $\Psi$, solving
\begin{equation}\label{translationflowall}
\partial_x \varphi(x) = \Psi(\varphi(x)).
\end{equation}
Conversely, since $\Psi$ is Lipschitz, a solution of \eqref{translationflowall} is uniquely determined by its initial data $\varphi(0)$, and then so is $W$. Therefore, the map $B$ given by \eqref{mapB} is injective.

\begin{proof}[Proof of Prop.~\ref{nonpausingprop}]
Separating the term $l=j$ in the product formula \eqref{greensfunctionproductformula}, we write
\[
G(x,x;z) = \sqrt{\frac{(\mu_j(x) - z)^2}{(E_j^+ - z)(E_j^- -z)}} \; \frac 12 \sqrt{ \frac 1{\underline E - z} \prod_{l\neq j} \frac{(\mu_l (x) -z)^2 }{(E_l^- - z)(E_l^+ - z)}  }.
\]
The product over $l\neq j$ has a finite limit as $z \to E$ since the $l$-th term is bounded by $1 + \gamma_l / \eta_{l,j}$. We can therefore conclude that
\begin{equation}\label{limitzeroorinfinity}
\lim_{\substack{z \in (E_j^-,E_j^+) \\ z \to E}} \lvert G(x,x;z)\rvert = \begin{cases} 0 & \mu_j(x) = E \\
+\infty & \mu_j(x) \neq E
\end{cases}
\end{equation}

We normalize the Weyl solutions for $z\in (E_j^-, E_j^+)$ by making them real-valued and by requiring
\[
\lvert \psi_\pm(0;z)\rvert^2 + \lvert \partial_x \psi_\pm(0;z)\rvert^2 = 1.
\]
We introduce the Pr\"ufer phase $\theta_\pm(x;z) \in \mathbb{R} / \pi \mathbb{Z}$ by
\[
\partial_x \psi_\pm  \cos \theta_\pm -  \psi_\pm\sin \theta_\pm =0.
\]
Using compactness, select a sequence $(E_j^-,E_j^+) \ni z_k \to E$ so that the Pr\"ufer phases at $x=0$ stabilize, i.e., the limits
\[
\lim_{k\to\infty} \theta_\pm (0;z_k)
\]
exist. By stability of solutions of ODEs, we can then conclude that
$\psi_\pm(x;z_k)$ converges in the $C^1$ sense on compacts in $x$ to solutions $\tilde\psi_\pm$ of the equation
\[
H_W \tilde \psi_\pm = E \tilde\psi_\pm.
\]
In particular, then,
\begin{equation}\label{limitthetatilde}
\lim_{k\to\infty} \theta_\pm(x;z_k) = \tilde \theta_\pm(x)
\end{equation}
with $\tilde\theta_\pm(x) \in \mathbb{R} / \pi \mathbb{Z}$ determined by
\[
\partial_x \tilde \psi_\pm  \cos \tilde \theta_\pm -  \tilde \psi_\pm\sin  \tilde\theta_\pm =0.
\]

The diagonal Green's function can be written as
\[
G(x,x;z) = \frac{ \cos \theta_-(x;z) \cos\theta_+(x;z)}{\sin(\theta_-(x;z) - \theta_+(x;z))}
\]
so, by \eqref{limitzeroorinfinity}, we have
\begin{equation}\label{limitzeroorinfinity2}
\lim_{\substack{z \in (E_j^-,E_j^+) \\ z \to E}} \left\lvert \frac{ \cos \theta_-(x;z) \cos\theta_+(x;z)}{\sin(\theta_-(x;z) - \theta_+(x;z))}  \right\rvert =  \begin{cases} 0 & \mu_j(x) = E \\
+\infty & \mu_j(x) \neq E
\end{cases}.
\end{equation}
Since $w\in (E_j^-, E_j^+)$ is above the bottom of the essential spectrum, $\psi_-(x;w)$ has infinitely many zeros, so we can pick $x$ such that $\mu_j(x) = w \in (E_j^-, E_j^+)$. Then, from \eqref{limitzeroorinfinity2} and \eqref{limitthetatilde} we conclude that $\tilde \theta_-(x) = \tilde \theta_+(x)$ for that value of $x$; but then that is true for all $x\in \mathbb{R}$, i.e.\ $\tilde \psi_+ = c \tilde \psi_-$ for some $c\in \{-1,+1\}$. By a renormalization we can assume that $\tilde\psi_-=\tilde\psi_+$ and denote that solution by $\tilde\psi$ from now on.

Now assume that $x$ is such that $\mu_j(x) = E$. From \eqref{limitzeroorinfinity2} and \eqref{limitthetatilde}, we conclude that $\cos \tilde\theta(x) = 0$, that is, $\tilde\psi(x) = 0$. We have therefore shown that
\[
D: = \{x \in \mathbb{R} \mid \mu_j(x) = E\} \subset \{ x\in \mathbb{R} \mid \tilde\psi(x) =0\} =: \tilde D.
\]
In particular, we conclude that $D$ is discrete, since $\tilde D$ is.

Since $\varphi_j$ is continuous, \eqref{translationflowj} holds away from a discrete set, and $\Psi_j > 0$, we conclude that $\varphi_j$ is a strictly increasing function of $x$. Therefore, $D$ strictly interlaces the set of zeros of $\psi_-(x;w)$ for  $w\in (E_j^-, E_j^+)$.

We will now show that $\tilde D$ strictly interlaces the same set, using the renormalized oscillation theory of Gesztesy--Simon--Teschl~\cite{GST}. By Theorem~1.3 of \cite{GST}, if $z, w$ are such that $(\min(z,w), \max(z,w))\cap S = \emptyset$, then the Wronskian of $\psi_-(x;w)$ and $\psi_+(x;z)$ is nonzero for all $x$. This is equivalent to saying that $\theta_-(x;w) - \theta_+(x;z) \notin \pi \mathbb{Z}$ for all $x$.

Thus, if we lift $\theta_\pm$ to continuous $\mathbb{R}$-valued functions, there is an $n\in \mathbb{Z}$ such that for all $x$,
\[
n\pi < \theta_-(x;w) - \theta_+(x;z) < (n+1)\pi.
\]
By continuity in $z\in (E_j^-,E_j^+)$, $n$ is independent of $z$. Taking $z =z_k$ and taking the limit $z_k\to E$ we obtain
\[
n\pi \le \theta_-(x;w) - \tilde\theta(x) \le (n+1)\pi
\]
However, by Prop.~3.4 of \cite{GST}, the function $\theta_-(x;w) - \tilde\theta(x)$ cannot have a local extremum where its value is in $\pi \mathbb{Z}$, so the previous estimates strengthen to
\[
n\pi < \theta_-(x;w) - \tilde\theta(x) < (n+1) \pi
\]
which means that the zeros of $\psi_-(x;w)$ and of $\tilde\psi(x)$ strictly interlace.

Knowing that $D\subset \tilde D$ and that each of those sets strictly interlaces the set of zeros of $\psi_-(x;w)$  proves that $D=\tilde D$, since $\psi_-(x;w)$ has zeros on any half-line in $\mathbb{R}$.

Finally, since \eqref{dirichletdatanonpausing} determines $\tilde \psi$ uniquely (up to a multiplicative factor), we conclude that the limit $\tilde\psi$ is independent of subsequence $z_k \to E$.
\end{proof}

\section{A Dubrovin-Type Formula for the KdV Flow}\label{sectionKdV}

In this section, we derive a formula for the time evolution of Dirichlet data for the KdV flow. This relies on ideas and results of Craig \cite{Cr}, discussed in the previous section, and of Rybkin \cite{Ry}, who derived the time evolution of the Weyl $M$-matrix for the KdV flow.

We will consider a solution $u(x,t): \mathbb{R}^2 \to \mathbb{R}$ of the initial value problem \eqref{KdV}, \eqref{KdVinitial}. As it has been mentioned above, by the Lax pair formalism, the spectrum $S = \sigma(H_{u(\cdot, t)})$ is independent of $t$. A much more recent result of Rybkin \cite{Ry} shows that the reflectionless property is preserved along the KdV flow. More precisely, \cite{Ry} assumes that $u(x,t)$ is a solution of \eqref{KdV} such that $u$, $\partial_x u$, $\partial_t u \in L^\infty(\mathbb{R} \times [0,T])$ for some $T>0$, and proves the following:
\begin{enumerate}[(i)]
\item The Weyl $M$-matrix is defined as
\[
M = \begin{pmatrix} \frac{m_- m_+}{m_- - m_+} & \frac 12 \frac{m_- + m_+}{m_- - m_+} \\
\frac 12 \frac{m_- + m_+}{m_- - m_+} & \frac{1}{m_- - m_+}
\end{pmatrix}
\]
where  $m_\pm(x,t;z)$ are the Weyl $m$-functions corresponding to $u(\cdot, t)$ on $[x, \pm \infty)$. The $M$-matrix evolves by
\begin{equation}\label{Rybkintimeevolution}
\partial_t M  = P M + M P^T
\end{equation}
where $P(x,t;z)$ is given by
\[
P = \begin{pmatrix}
\partial_x u   & 2 (u-z)(u+2z) - \partial_{x}^2 u \\
2 (u+2z)  & - \partial_x u
\end{pmatrix}
\]
This time evolution is valid for all $z\in \mathbb{C} \setminus \sigma(H_V)$. It is also valid in the sense of the nontangential limit $M(x,t;z+i0)$ for those $z\in \sigma(H_V)$ for which that nontangential limit exists for $t=0$.
\item For Lebesgue-a.e.\ $z \in \sigma_\ac(H_V)$, the value of the reflection coefficient
\[
\left \lvert \frac{\overline{m}_+ - m_-}{m_+ - m_-} \right\rvert
\]
is independent of $t\in [0,T]$.
\end{enumerate}

The condition that the reflection coefficient be equal to $0$ for a.e.\ $z \in S$ is equivalent to reflectionlessness of the potential; see \cite{SY}. In particular, Rybkin's results imply that if $V \in \mathcal{R}(S)$ and $S = \sigma_\ac(H_V)$, then $u(\cdot,t) \in \mathcal{R}(S)$ for all $t\in [0,T]$.

Since $u(\cdot, t)$ is reflectionless for all $t$, we can introduce the corresponding Dirichlet data $\varphi(x,t)$ as in the previous section.

Recall from the previous section the torus $\mathcal{D}(S)$ and, on it, the scalar field $Q_1$ and vector field $\Psi$. We now introduce the vector field $\Xi$ on $\mathcal{D}(S)$, with components
\[
\Xi_j(\varphi) =  -2 (Q_1 + 2 \mu_j ) \Psi_j.
\]

\begin{prop}
If $S$ obeys \eqref{Craig2}, then $\Xi$ is a Lipshitz vector field on $\mathcal{D}(S)$.
\end{prop}

\begin{proof}
We begin by noting that each component $\Xi_j$ is bounded as
\[
\lVert \Xi_j \rVert_\infty \le 4 (\lVert Q_1 \rVert_\infty + 2 \lvert \underline E \rvert + 2 \eta_{j,0} ) C_j
\]
Next, each component $\Xi_j$ has partial derivatives given by
\begin{align*}
\frac{\partial \Xi_j}{\partial \varphi_k} & = \left(  \frac {Q_1 + 2 \mu_j}{\mu_k - \mu_j} - 2 \right) \Psi_j  \gamma_k \sin\varphi_k, \qquad k\neq j \\
\frac{\partial \Xi_j}{\partial \varphi_j} & = - \frac 14 \left( \frac 1{\underline E - \mu_j} + \sum_{k\neq j} \left( \frac 1{E_k^- - \mu_j} + \frac 1{E_k^+ - \mu_j} - \frac 2{\mu_k - \mu_j} \right) \right) \Xi_j  \gamma_j \sin \varphi_j
\end{align*}
from which it is easy to estimate
\begin{align*}
\left\lvert \frac{\partial \Xi_j}{\partial \varphi_k}\right\rvert & \le 2 \left( \frac {\lVert Q_1\rVert_\infty + 2 \lvert \underline E \rvert +2\eta_{j,0}}{\eta_{j,k} } + 2 \right) C_j  \gamma_k, \qquad k\neq j \\
\left\lvert \frac{\partial \Xi_j}{\partial \varphi_j} \right\rvert & \le  \frac 12 \left( \frac 1{\eta_{j,0}} + \sum_{k\neq j} \frac{\gamma_k}{\eta_{j,k} (\eta_{j,k} + \gamma_k) } \right) \left( \lVert Q_1\rVert_\infty + 2 \lvert \underline E \rvert +2\eta_{j,0}  \right) C_j  \gamma_j
\end{align*}
and then
\begin{align*}
\lVert \Xi(\varphi) - \Xi(\tilde \varphi) \rVert & = \sup_{j\in J} \gamma_j^{1/2} \lVert \Xi_j(\varphi) - \Xi_j(\tilde\varphi) \rVert \\
& \le \sup_{j\in J} \gamma_j^{1/2} \sum_{k\in J} \left\lVert \frac{\partial \Xi_j}{\partial \varphi_k}\right\rVert_\infty  \lVert \varphi_k - \tilde\varphi_k \rVert \\
&  \le   \lVert \varphi - \tilde\varphi \rVert \sup_{j\in J} \sum_{k\in J}  \gamma_j^{1/2} \gamma_k^{-1/2}  \left\lVert \frac{\partial \Xi_j}{\partial \varphi_k}\right\rVert_\infty
\end{align*}
It remains to show that the $\sup$ is finite. Splitting off the $k=j$ term from the sum, the remainder can be estimated by
\[
\sup_{j\in J} \sum_{k\neq j} 2 \gamma_j^{1/2} \gamma_k^{1/2} \left(  \frac{\lVert Q_1\rVert_\infty + 2 \lvert \underline E \rvert + 2 \eta_{j,0}}{\eta_{j,k}} + 2 \right) C_j \le \sup_{j\in J}  C \left( \sum_{k\neq j} \frac{ \gamma_j^{1/2} \gamma_k^{1/2}}{\eta_{j,k}} (1+ \eta_{j,0}) C_j +  \gamma_j^{1/2} C_j \sum_{k\neq j} \gamma_k^{1/2}\right) < \infty
\]
and the $k=j$ term can be estimated by
\[
\sup_j \left( \frac 1{\eta_{j,0}} + \sum_{k\neq j} \frac{\gamma_k}{\eta_{j,k} (\eta_{j,k} + \gamma_k) } \right) \left( \lVert Q_1\rVert_\infty + 2 \lvert \underline E \rvert  +2\eta_{j,0}  \right) C_j  \gamma_j < \infty
\]
from the conditions \eqref{Craig2}.
\end{proof}

\begin{prop}\label{propKdVflow}
Let $u(x,t)$ be a solution of \eqref{KdV}, \eqref{KdVinitial} such that $u$, $\partial_x u$, $\partial_t u \in L^\infty(\mathbb{R} \times [0,T])$ for some $T>0$. Assume that $V\in \mathcal{R}(S)$ and $S = \sigma_\ac(H_V)$. The function $\varphi(x,t)$ is jointly continuous in $(x,t)$; it is differentiable in $t$ and obeys
\begin{equation}\label{KdVflowangularDirichlet}
\partial_t \varphi_j(x,t) = \Xi_j(\varphi(x,t)).
\end{equation}
\end{prop}

As a first step, we prove a version of this proposition away from gap edges.

\begin{lemma}\label{diffawayfromgapedges}
Under the assumptions of Prop.~\ref{propKdVflow}, at any $(x,t)$ such that $\varphi_j(x,t) \notin \pi \mathbb{Z}$, $\varphi_j(x,t)$ is differentiable in $t$ and obeys \eqref{KdVflowangularDirichlet}.
\end{lemma}

\begin{proof}
The $M$-matrix is analytic in $z\in \mathbb{C} \setminus S$. It has a removable singularity at a Dirichlet eigenvalue, where it is equal to
\begin{equation}\label{Mmatrixmuj}
M(x,t;\mu_j(x,t)) =  \begin{pmatrix} -\sigma_j(x,t)  m_{-\sigma_j(x,t)}(\mu_j(x,t)) & - \frac 12 \sigma_j(x,t) \\
- \frac 12 \sigma_j(x,t) & 0
\end{pmatrix}.
\end{equation}
Note that the bottom right entry of the $M$-matrix is precisely the diagonal Green's function. From \eqref{Rybkintimeevolution}, the time evolution of the bottom right entry is
\[
\partial_t M_{2,2}(x,t;z) = - 2 (\partial_x u)(x,t)  M_{2,2}(z,t) + 4 (u(x,t) + 2 z)  M_{1,2}(x,t;z).
\]
In particular, for $z=\mu_j(x,t)$, from \eqref{Mmatrixmuj},
\begin{equation}\label{dtM22}
\left. \partial_t M_{2,2}(x,t;z) \right\rvert_{z=\mu_j(x,t)} = - 2 \sigma_j(x,t) (u(x,t) + 2 \mu_j(x,t)).
\end{equation}
At any $(x,t)$ such that $\mu_j(x,t) \in (E_j^-, E_j^+)$, by the implicit function theorem, $\mu_j(x,t)$ is differentiable and
\[
\frac{\partial\mu_j}{\partial t} = - \frac{ \partial G / \partial t\vert_{z = \mu_j}}{ \partial G / \partial z \vert_{z = \mu_j}}.
\]
Using \eqref{dtM22} and \eqref{greensfunctionderivativeatmu}, this becomes
\begin{equation}\label{dmujdt}
\frac{\partial\mu_j}{\partial t} =  4 \sigma_j (u + 2 \mu_j ) \left( (\underline{E} - \mu_j)(E_j^- - \mu_j)(E_j^+ - \mu_j)  \prod_{l\neq j} \frac{(E_l^- - \mu_j)(E_l^+ - \mu_j ) }{(\mu_l  - \mu_j)^2}  \right)^{1/2}.
\end{equation}
Combining this with the trace formula $u=Q_1 \circ \varphi$, we obtain an expression for $\partial \mu_j / \partial t$ in terms of the Dirichlet data.

This shows that $\mu_j(x,t)$ is continuous in $(E_j^-, E_j^+)$. By looking at the time evolution of the off-diagonal entries of $M(x,t;\mu_j(x,t))$, we also see that $\sigma_j(x,t)$ is constant in $t$ as long as $\mu_j(x,t) \in (E_j^-, E_j^+)$.

Then, from \eqref{mufromvarphi}, \eqref{sigmafromvarphi}, we conclude that $\varphi_j(x,t)$ is also differentiable in $t$ and
\[
\frac{\partial\mu_j}{\partial t} = - \frac 12 (E_j^+ - E_j^-) \sin \varphi_j \frac{\partial \varphi_j}{\partial t} = - \sigma_j \sqrt{(E_j^+ - \mu_j)(\mu_j - E_j^-)} \; \frac{\partial\varphi_j}{\partial t}
\]
which, when combined with \eqref{dmujdt} and solved for $\frac{d\varphi_j}{dt}$, concludes the proof.
\end{proof}

\begin{lemma}
$\varphi_j(x,t)$ is continuous in $t$.
\end{lemma}

\begin{proof}
Due to the previous lemma, we only need to address continuity of $\varphi_j$ at points with $\mu_j \in \{E_j^-, E_j^+\}$. Let us assume $\mu_j(x_0,t_0) = E_j^+$; the other case is handled analogously.

For any $\epsilon > 0$, $G(x_0,x_0; E_j^+ - \epsilon, t_0) < 0$. By \eqref{Rybkintimeevolution}, the diagonal Green's function is continuous in $t$; thus, there is a $\delta > 0$ such that $G(x_0,x_0; E_j^+ - \epsilon, t) < 0$ for $\lvert t - t_0 \rvert < \delta$. Thus, $\mu_j(x_0,t) > E_j^+ - \epsilon$ for $\lvert t- t_0 \rvert < \delta$.
\end{proof}

Before we prove that $\varphi_j$ is differentiable everywhere, as an intermediate step, we have to prove statements which will eventually be superseded.

\begin{lemma} Fix $x\in \mathbb{R}$ and $j\in J$. Lift $\varphi_j(x,t)$ to a continuous function $\varphi_j : \mathbb{R}^2 \to \mathbb{R}$.
\begin{enumerate}[(i)]
\item For any $a<b$,
\begin{equation}\label{upperboundvarphim}
\lvert \varphi_j(x,b) - \varphi_j(x,a)  \rvert \le \int_a^b \lvert \Xi_j(\varphi(x,\tau)) \rvert d\tau.
\end{equation}
\item For any $a<b$, if $\Xi_j(\varphi(x,\tau)) \ge 0$ for $\tau \in (a,b)$, then
\begin{equation}\label{monotonicityvarphim}
0 \le  \varphi_j(x,b) -  \varphi_j(x,a)  \le \int_a^b  \Xi_j(\varphi(x,\tau)) d\tau.
\end{equation}
\end{enumerate}
\end{lemma}

\begin{proof}
We supress the parameter $x$ throughout this proof, since it can be considered fixed.

By uniform continuity, the interval $[a,b]$ can be split into finitely many closed intervals $I$ with the property
\begin{equation}\label{apriorivarphim}
\lvert \varphi_j(u) - \varphi_j(v) \rvert < \pi\text{ if }(u,v) \subset I.
\end{equation}
By the triangle inequality, it suffices to prove \eqref{upperboundvarphim}, \eqref{monotonicityvarphim} for $(a,b) \subset I$ with an interval $I$ with the property \eqref{apriorivarphim}.

If $\varphi_j(\tau) \notin \pi \mathbb{Z}$ for $\tau \in (a,b)$, then \eqref{KdVflowangularDirichlet} holds on $(a,b)$, so
\[
\varphi_j(b) - \varphi_j(a) = \int_a^b \Xi_j(\varphi(\tau)) d\tau,
\]
and \eqref{upperboundvarphim},  \eqref{monotonicityvarphim} follow. Otherwise, take
\[
a_1 = \inf \{ \tau \in (a,b) \vert \varphi_j(\tau) \in \pi \mathbb{Z} \}, \qquad b_1 = \sup \{ \tau \in (a,b) \vert \varphi_j(\tau) \in \pi \mathbb{Z} \}.
\]
Then  \eqref{KdVflowangularDirichlet} holds on $(a,a_1)$ and $(b_1,b)$; moreover, by \eqref{apriorivarphim} and $\varphi_j(a_1), \varphi_j(b_1) \in \pi \mathbb{Z}$, we have $\varphi_j(a_1) = \varphi_j(b_1)$, so
\begin{align*}
\varphi_j(b) - \varphi_j(a) & = \varphi_j(b) - \varphi_j(b_1)  + \varphi_j(a_1) - \varphi_j(a)  = \int_a^{a_1}  \Xi_j(\varphi(\tau))  d\tau + \int_{b_1}^b  \Xi_j(\varphi(\tau))  d\tau
\end{align*}
which implies \eqref{upperboundvarphim},  \eqref{monotonicityvarphim}.
\end{proof}

\begin{proof}[Proof of Proposition~\ref{propKdVflow}]

For differentiability, given Lemma~\ref{diffawayfromgapedges}, there are two cases left to consider.

First case: $\varphi_j(x_0,t_0) \in \pi\mathbb{Z}$ and $\Xi_j(\varphi(x_0,t_0)) = 0$. Since $\lim_{\tau\to t_0} \Xi_j(\varphi(x_0,\tau)) = \Xi_j(\varphi(x_0,t_0)) = 0$, \eqref{upperboundvarphim} implies that
\[
\lim_{t\to t_0} \frac{\varphi_j(x_0,t) - \varphi_j(x_0,t_0) }{t-t_0}  = 0
\]
so $\varphi_j(x_0,t)$ is differentiable at $t=t_0$ and $\partial_t \varphi_j(x_0,t_0) = 0 = \Xi_j(\varphi(x_0,t_0))$.

Second case: $\varphi_j(x_0,t_0) \in \pi\mathbb{Z}$ and $\Xi_j(\varphi(x_0,t_0)) \neq 0$. Without loss of generality we assume $\Xi_j(\varphi(x_0,t_0)) > 0$ (the other case is treated analogously). There is then a neighborhood $(x,t) \in (x_0-\epsilon, x_0+\epsilon) \times (t_0-\epsilon,t_0+\epsilon)$ in which $\Xi_j(\varphi(x,t)) \ge 0$. We can then conclude that
\[
\varphi_j(x_0,t_0) < \varphi_j(x,t) < \varphi_j(x_0,t_0)+\pi,\quad\text{ if }x \in (x_0,x_0+\epsilon), t \in [t_0,t_0+\epsilon);
\]
the first inequality follows by the strict monotonicity of the translation flow and \eqref{monotonicityvarphim}, and the second by joint continuity in $(x,t)$ and shrinking $\epsilon$ if necessary.

For $x\in (x_0,x_0+\epsilon)$ and $t \in (t_0,t_0+\epsilon)$, we can therefore use \eqref{KdVflowangularDirichlet} to conclude that
\begin{equation}\label{varphimintegral2}
\varphi_j(x,t) - \varphi_j(x,t_0) = \int_{t_0}^t \Xi_j(\varphi(x,\tau)) d\tau
\end{equation}
for $x \in (x_0,x_0+\epsilon)$. Letting $x \to x_0$, we conclude that \eqref{varphimintegral2} holds also for $x=x_0$. It then follows that
\[
\lim_{t \to t_0+} \frac{\varphi_j(x_0,t) - \varphi_j(x_0,t_0) }{t-t_0} = \Xi_j(\varphi(x_0,t_0)).
\]
A similar argument starting with $x \in (x_0-\epsilon,x_0)$ and $t \in (t_0-\epsilon,t_0)$ provides the left derivative, which completes the proof.
\end{proof}

\section{Construction of Reflectionless Potentials and KdV Solutions by Finite Gap Approximants}\label{sec.4}

In this section, we construct solutions to two questions. The first result is a converse of Craig's results, where we show that any Dirichlet data $f\in \mathcal{D}(S)$ corresponds to a reflectionless operator $H_W$ with spectrum $S$; in other words, we show that the map $B$ defined by \eqref{mapB} is onto. In the second result, we construct solutions of the KdV equation with a prescribed spectrum $S$ and prescribed initial Dirichlet data.

On the torus $\mathcal{D}(S)$ of (angular) Dirichlet data given as before, we introduce for $k=1,2,3$,
\begin{equation}\label{Qkdefinition}
Q_k(\varphi) = \underline E^{k} + \sum_{j\in J} ((E_j^-)^k + (E_j^+)^k - 2(\mu_j)^k).
\end{equation}

\begin{lemma}\label{L41}
If $S$ obeys \eqref{traceconditions}, then $Q_1$, $Q_2$, $Q_3$ are continuous scalar fields on $\mathcal{D}(S)$.
\end{lemma}

\begin{proof}
Since $\mathcal{D}(S)$ is equipped with the product topology, it is sufficient to establish that the series in \eqref{Qkdefinition} converges absolutely and uniformly. For $k=1$, this follows from
\[
\lvert E_j^- + E_j^+ - 2 \mu_j \rvert \le \gamma_j.
\]
For $k\ge 2$, note that for $x\in [E_j^-, E_j^+]$,
\[
\lvert x^k - (E_j^-)^k \rvert
\le \sum_{m=1}^{k} \binom{k}{m} \lvert E_j^-\rvert^{k-m} (x-E_j^-)^{m}
\le \sum_{m=1}^{k} \binom{k}{m} \lvert E_j^-\rvert^{k-m} \gamma_j^{m}.
\]
Writing $E_j^- = \underline E + \eta_{j,0}$ and expanding,
\[
\lvert x^k - (E_j^-)^k \rvert \le \sum_{m=1}^{k} \sum_{n=0}^{k-m} \binom{k}{m} \binom{k-m}{n} \lvert \underline E\rvert^{k-m-n} \eta_{j,0}^{n} \gamma_j^{m}.
\]
By the arithmetic-geometric inequality,
\[
\eta_{j,0}^{n} \gamma_j^{m} \le \Gamma^{m-1} \left( \frac{n}{k-1} \eta_{j,0}^{k-1} \gamma_j + \frac{k-1-n}{k-1}\gamma_j \right)
\]
where $\Gamma =\sup_{j\in J} \gamma_j$. We therefore obtain
\[
\lvert x^k - (E_j^-)^k \rvert \le C (\eta_{j,0}^{k-1} \gamma_j + \gamma_j)
\]
where $C$ is a constant depending only on $k$, $\underline E$, and $\Gamma$. Thus,
\[
\lvert (E_j^-)^k + (E_j^+)^k - 2(\mu_j)^k \rvert \le 3 C (\eta_{j,0}^{k-1} \gamma_j + \gamma_j)
\]
and the series converges absolutely and uniformly by \eqref{Qkdefinition}.
\end{proof}

Our first result is a converse of Craig's results: we show that any Dirichlet data in $\mathcal{D}(S)$ corresponds to a reflectionless potential $\mathcal{R}(S)$, or in other words, that the map $B: \mathcal{R}(S) \to \mathcal{D}(S)$ defined in \eqref{mapB} is onto. Combined with Craig's theory, this shows that $B$ is a homeomorphism.

For this result, we will assume that $S$ is equal to its essential closure, i.e., that
\begin{equation}\label{acproperty}
\forall x\in S \quad \forall \epsilon > 0 \quad \Leb(S \cap (x-\epsilon,x+\epsilon)) > 0,
\end{equation}
where $\Leb$ denotes Lebesgue measure. This certainly holds if, for instance, $S$ is the a.c.\ spectrum of some self-adjoint operator, so it will hold in our applications.

\begin{prop}\label{P4a}
Let $S$ obey \eqref{Craig1}, \eqref{acproperty}. Let $f \in \mathcal{D}(S)$ and let $\phi: \mathbb{R} \to \mathcal{D}(S)$ be the unique solution of the initial value problem
\begin{align*}
\partial_x \phi(x) & = \Psi(\phi(x)), \\
\phi(0) & = f.
\end{align*}
If we define $W: \mathbb{R} \to \mathbb{R}$ by
\begin{equation}\label{Q0W}
W = Q_1 \circ \phi,
\end{equation}
then $W\in\mathcal{R}(S)$ and $B(W)=f$. Moreover, if $S$ also obeys \eqref{traceconditions}, then $W \in C^4(\mathbb{R}) \cap W^{4,\infty}(\mathbb{R})$ and $W$ obeys the higher order trace formulas
\begin{align}
Q_2 \circ \phi & = - \tfrac 12 \partial_{x}^2 W + W^2 \label{Q1W} \\
Q_3 \circ \phi & =  \frac 3{16} \partial_{x}^4 W - \frac 32 W \partial_{x}^2 W - \frac{15}{16} (\partial_x W)^2 + W^3 \label{Q2W}
\end{align}
\end{prop}

\begin{remark}
The claim that $B(W)=f$ is not trivial; since $Q_1$ is far from injective, it is not a priori obvious that two different trajectories $\phi, \tilde\phi$ under the flow $\Psi$ may not lead to the same potential, $Q_1 \circ \phi = Q_1 \circ \tilde \phi$. The above proposition rules this out.
\end{remark}

The other main result of this section concerns existence of a solution of the KdV equation with prescribed initial Dirichlet data.

\begin{prop}\label{P4}
Let $S$ obey \eqref{Craig2}, \eqref{traceconditions}, \eqref{acproperty}. Let $f \in \mathcal{D}(S)$. There exists a function $\varphi: \mathbb{R}^2 \to \mathcal{D}(S)$ such that $\varphi(0,0) = f$ and
\[
\partial_x \varphi(x,t) = \Psi(\varphi(x,t)), \qquad \partial_t \varphi(x,t) = \Xi(\varphi(x,t)).
\]
If we define $u: \mathbb{R}^2 \to \mathbb{R}$ by
\begin{equation}\label{Q0u}
u = Q_1 \circ \varphi,
\end{equation}
then the function $u(x,t)$ obeys the KdV equation \eqref{KdV}. Moreover, for each $t\in \mathbb{R}$, we have $u(\cdot, t) \in \mathcal{R}(S)$ and $B(u(\cdot,t))=\varphi(0,t)$.
\end{prop}

We will prove these propositions by a Cauchy sequence argument. The main ingredient is the theory of finite gap potentials, already mentioned in the introduction.

In this and the next section, we will find it notationally convenient to assume that gaps are indexed by positive integers, i.e., that $J=\mathbb{N}$. This doesn't reduce generality, since in the special case where $J$ is finite, the results are already known. Denote
\[
S^N = [\underline E, \infty) \setminus \bigcup_{j\le N} (E_j^-, E_j^+)
\]
and denote the corresponding isospectral torus by $\mathcal{D}(S^N)$ and vector fields corresponding to translation and KdV flows by $\Psi^N$ and $\Xi^N$.

Starting from an element $f \in \mathcal{D}(S)$, we observe the functions $\varphi^N:\mathbb{R}^2 \to \mathcal{D}(S^N)$ which solve the equations
\begin{equation}
\partial_x \varphi^N = \Psi^N(\varphi^N), \qquad \partial_t \varphi^N = \Xi^N(\varphi^N)
\end{equation}
and obey the initial conditions
\[
\varphi^N_j(0,0) = f_j, \qquad j\le N.
\]

This uniquely determines the functions $\varphi^N$, since the vector fields $\Psi^N$ and $\Xi^N$ commute, which follows from a direct calculation that for all $i,j \le N$,
\[
\Psi^N_j \partial_j \Xi^N_i - \Xi^N_j \partial_j  \Psi^N_i = 0.
\]
At this point we recall that, by \cite[Theorem 1.48, Lemma 1.16]{GH}, the function $u_N = Q_1^N \circ \varphi^N: \mathbb{R}^2 \to \mathbb{R}$ has the following properties:
\begin{enumerate}[(i)]
\item For every $t$, the function $u_N(x,t)$ is almost periodic in $x$ and the Schr\"odinger operator $H_{u_N(\cdot,t)}$ has spectrum  $S^N$
\item $u_N$ solves the KdV equation \eqref{KdV}
\item The Dirichlet data of $u_N$ are $\varphi^N$
\item The following trace formulas hold:
\begin{align}
Q_1^N \circ \varphi^N & = u_N \label{Q0} \\
Q_2^N \circ \varphi^N & = - \tfrac 12 \partial_{x}^2 u_N + u_N^2 \label{Q1} \\
Q_3^N \circ \varphi^N & =  \frac 3{16} \partial_{x}^4 u_N - \frac 32 u_N \partial_{x}^2 u_N - \frac{15}{16} (\partial_x u_N)^2 + u_N^3 \label{Q2}
\end{align}
where $Q_k^N$ are scalar fields on $\mathcal{D}(S^N)$ given by the expression
\[
Q^N_k(\varphi) = \underline E^{k} + \sum_{j=1}^N ((E_j^-)^k + (E_j^+)^k - 2(\mu_j)^k).
\]
\end{enumerate}

We begin with a statement about stability of trajectories of a vector field.

\begin{lemma}\label{L43}
Assume that $U$, $\tilde U$ are Lipshitz vector fields on $\mathcal{D}(S)$, with Lipshitz constants less or equal to $L$. Consider solutions $\phi, \tilde \phi: \mathbb{R} \to \mathcal{D}(S)$ of $\partial_x \phi = U(\phi)$, $\partial_x \tilde\phi = \tilde U(\tilde \phi)$. Then
\begin{equation}\label{exponentialestimate}
\lVert \phi(x) - \tilde \phi(x) \rVert_{\mathcal{D}(S)} \le 2 \left(\lVert \phi(0) - \tilde\phi(0) \rVert_{\mathcal{D}(S)} + C \right) e^{m \lvert x\rvert},
\end{equation}
where $m = 2 L \ln 2$ and
\[
C = \frac 1L \sup_{j\in J} \gamma_{j}^{1/2} \min(2\pi, \lVert U_j - \tilde U_j \rVert).
\]
\end{lemma}

\begin{proof}
Let $l = 1/(2L)$ and define the map $\Lambda$ from $C([0,l], \mathcal{D}(S))$ to itself by
\[
\Lambda(g)(x) = f + \int_0^x  U(g(y)) dy.
\]
Then
\[
\lVert \Lambda(g)(x) - \Lambda(h)(x) \rVert \le \int_0^x \lVert U(g(y)) - U(h(y)) \rVert dy \le L x \lVert g - h \rVert
\]
so, by our choice of $l$,
\[
\lVert \Lambda(g) - \Lambda(h) \rVert \le \frac 12 \lVert g - h\rVert.
\]
Since $\Lambda$ is a contraction with coefficient $1/2$, it has a unique fixed point, which is precisely the path $\phi(x)$. Similarly, defining $\tilde\Lambda$ on $C([0,l], \mathcal{D}(S))$ by
\[
 \tilde\Lambda(g)(x) = \tilde f + \int_0^x \tilde U(g(y)) dy,
\]
$\tilde \Lambda$ is a contraction with coefficient $1/2$ whose unique fixed point is $\tilde\phi(x)$.

There are two ways of estimating $\Lambda_j(g) - \tilde\Lambda_j(g)$: one is by integrating,
\[
\lVert \Lambda_j(g) - \tilde\Lambda_j(g) \rVert \le \lVert f_j - \tilde f_j\rVert + l \lVert U_j - \tilde U_j\rVert,
\]
and the other is the trivial estimate $\lVert \Lambda_j(g) - \tilde\Lambda_j(g) \rVert \le \pi$. Using the smaller of these two estimates, multiplying by $\gamma_j^{1/2}$, and taking $\sup_{j\in J}$, we conclude
\[
\lVert \Lambda(g) - \tilde\Lambda(g) \rVert \le \lVert f - \tilde f\rVert + \frac C2.
\]
In particular, for $g= \tilde\phi$,
\[
\lVert \Lambda(\tilde\phi) - \tilde \phi \rVert \le \lVert f - \tilde f\rVert + \frac C2.
\]
By standard estimates from fixed point theorems, since $\phi = \lim_{n\to\infty} \Lambda^n(\tilde \phi)$ and $\Lambda$ is a $1/2$-contraction, this implies that
\[
\lVert \phi - \tilde\phi \rVert \le  2 \lVert f - \tilde f\rVert + C,
\]
that is,
\[
\lVert \phi(x) - \tilde\phi(x) \rVert \le 2 \lVert \phi(0) - \tilde\phi(0) \rVert + C, \qquad \forall x \in [0,l].
\]
Applying this inductively, we see that for $n\in \mathbb{N}$,
\[
\lVert \phi(nl) - \tilde\phi(nl) \rVert \le 2^n \left(\lVert \phi(0) - \tilde\phi(0) \rVert +  C \right)   - C
\]
and therefore for $x>0$,
\[
\lVert \phi(x) - \tilde\phi(x) \rVert \le 2^{1+x/l} \left(\lVert \phi(0) - \tilde\phi(0) \rVert + C\right)   - C.
\]
The same argument works for negative $x$ by reversing the vector fields, so we conclude \eqref{exponentialestimate}.
\end{proof}

To apply the previous lemma to compare trajectories on different isospectral tori, we will lift all the fields and solutions to $\mathcal{D}(S)$ in the following way.

There is a natural projection $\pi:\mathcal{D}(S) \to \mathcal{D}(S^N)$, given by
\[
\pi(\phi)_j = \phi_j, \quad j \le N.
\]
To lift the scalar fields, simply define $\tilde Q^N_k: \mathcal{D}(S) \to \mathbb{R}$ by
\[
\tilde Q_k^N = Q_k^N \circ \pi.
\]

Introduce the vector field $\tilde\Psi^N$ on $\mathcal{D}(S)$ by
\[
\tilde\Psi^N_j(\phi) = \begin{cases}
\Psi^N_j(\pi(\phi)) & j \le N \\
\Psi_j(\pi(\phi)) & j > N
\end{cases}
\]
and, analogously, define the vector field $\tilde\Xi^N$ on $\mathcal{D}(S)$ by
\[
\tilde\Xi^N_j(\phi) = \begin{cases}
\Xi^N_j(\pi(\phi)) & j \le N \\
\Xi_j(\pi(\phi)) & j > N
\end{cases}
\]
To define $\tilde\varphi^N: \mathbb{R}^2 \to \mathcal{D}(S)$, set the initial value
\[
\tilde\varphi^N(0,0) = f,
\]
determine its values for $x=0$ by requiring
\[
\partial_t \tilde\varphi^N(0,t) =  \tilde \Xi^N(\tilde\varphi^N(0,t)), \quad \forall t\in\mathbb{R},
\]
and then its values for arbitrary $x$ by requiring
\[
\partial_x \tilde\varphi^N(x,t) = \tilde\Psi^N(\tilde\varphi^N(x,t)), \quad \forall x,t\in\mathbb{R}.
\]
With these definitions, obviously,
\[
\varphi^N = \pi \circ \tilde \varphi^N.
\]

We similarly introduce $\varphi:\mathbb{R}^2 \to \mathcal{D}(S)$ by
\[
\varphi(0,0) = f,
\]
\[
\partial_t \varphi(0,t) =   \Xi(\varphi(0,t)), \quad \forall t\in\mathbb{R},
\]
\[
\partial_x \varphi(x,t) = \Psi(\varphi(x,t)), \quad \forall x,t\in \mathbb{R}.
\]

\begin{lemma}\label{L44b}
Assume that the set $S$ obeys \eqref{Craig2}. There exists $m>0$ and constants $K_N$ such that $\lim_{N\to \infty} K_N = 0$ and, for all $x,t \in \mathbb{R}$, and all $j\le N$,
\begin{equation}\label{expestimateKM}
 \lVert \tilde\varphi^N(x,t) - \varphi(x,t) \rVert\le K_N e^{m\lvert x\rvert + m \lvert t\rvert}.
\end{equation}
\end{lemma}

\begin{proof}[Proof of Lemma~\ref{L44b}]
If \eqref{Craig2} holds, we know that $\Psi$, $\Xi$ are Lipshitz vector fields; the proof of these facts gives explicit upper bounds for the Lipshitz constants in terms of gap sizes and distances, so it also applies to $\tilde\Psi^N$ and $\tilde \Xi^N$, giving uniform Lipshitz estimates in $N$. Denote by $L$ such an upper bound on the Lipshitz constant which works for $\Psi$, $\Xi$, and all the $\tilde\Psi^N$, $\tilde \Xi^N$.

In particular, by the Lipshitz property, values of $\varphi(x,t)$ and $\tilde\varphi^N(x,t)$ are uniquely determined by the above definitions.

For $j\le N$,
\begin{align*}
\left\lvert \Psi_j - \tilde \Psi_j^N \right\rvert & \le \lvert \tilde\Psi_j^N \rvert \left\lvert \left\lvert \prod_{k=N+1}^\infty \frac{(E_k^- - \mu_j)(E_k^+ - \mu_j)} {(\mu_k - \mu_j)^2}  \right\rvert^{1/2} - 1 \right\rvert \\
& \le 2 (\eta_{j,0}+\gamma_j)^{1/2} \prod_{\substack{1 \le k \le N \\ k\neq j}} \left( 1 + \frac{\gamma_k}{\eta_{j,k}} \right)^{1/2}  \left( \prod_{k=N+1}^\infty \left( 1 + \frac{\gamma_k}{\eta_{j,k}} \right)^{1/2} - 1 \right) \\
& \le 2(C_j - C_{j,N})
\end{align*}
where $C_j$ is defined in \eqref{Cjdef} and
\[
C_{j,N} = (\eta_{j,0} + \gamma_j)^{1/2} \prod_{\substack{l \le N \\ l\neq j}} \left( 1 + \frac{\gamma_l}{\eta_{j,l}} \right)^{1/2}.
\]
Since $\tilde\Xi^N_j(\varphi) =  -2 (\tilde Q^N_1 + 2 \mu_j ) \tilde \Psi^N_j$, we estimate
\begin{align*}
\left\lvert \Xi_j - \tilde \Xi_j^N \right\rvert & \le 2 \lvert  Q_1 - \tilde Q_1^N \rvert \, \lvert \tilde \Psi_j^N \rvert + 2 \lvert Q_1 + 2 \mu_j \rvert\,  \lvert  \Psi_j - \tilde\Psi_j^N \rvert \\
& \le  (D_j - D_{j,N}) C_{j,N} +  D_j (C_j - C_{j,N}) \\
& \le D_j C_j - D_{j,N} C_{j,N}
\end{align*}
where
\[
D_{j,N} = 2\sum_{l=1}^N \gamma_l + 4 \eta_{j,0} + 4\lvert \underline E\rvert, \qquad D_j =\lim_{N\to\infty} D_{j,N}.
\]
Let $m = 2 L \ln 2$ and let
\[
K_N =  \frac 8L  \sup_{j\le N} \gamma_j^{1/2} \max\left( 2\pi, C_j - C_{j,N}, D_j C_j - D_{j,N} C_{j,N} \right).
\]
Since $\lim_{j\to \infty} \gamma_j =0$ and, for each $j$, $C_{j,N} \to C_j$, $D_{j,N} \to D_j$ as $N\to \infty$, it is clear that $\lim_{N\to\infty} K_N = 0$.

Applying Lemma~\ref{L43} to compare trajectories of vector fields $\tilde \Xi^N$ and $\Xi$ with initial condition $f$, we conclude
\begin{equation}\label{estimatevarphi0t}
 \lVert \varphi^N(0,t) - \tilde\varphi(0,t) \rVert_{\mathcal{D}(S)} \le \tfrac 14 K_N e^{m\lvert t \rvert}.
\end{equation}
Applying Lemma~\ref{L43} to compare trajectories of $\tilde\Psi^N$ and $\Psi$ using \eqref{estimatevarphi0t} as initial conditions, we obtain
\[
 \lVert \varphi(x,t) - \tilde\varphi^N(x,t) \rVert_{\mathcal{D}(S)} \le 2 ( \tfrac 14 K_N e^{m\lvert t \rvert} + \tfrac 14 K_N ) e^{m\lvert x \rvert}
\]
which implies \eqref{expestimateKM}.
\end{proof}

If we remove all discussion of $t$-dependence and of the vector fields $\Xi$, $\Xi^N$ from the previous proof, then we only need the Lipshitz property to hold for $\Psi$, $\Psi^N$, and the argument only requires the weaker condition \eqref{Craig1} instead of \eqref{Craig2}. Likewise, the constant $K_N$ from the previous proof can be redefined as
\[
K_N = \frac 8L  \sup_{j\le N} \gamma_j^{1/2} \min(2\pi, C_j - C_{j,N}).
\]
 The argument then yields the following result.

\begin{lemma}\label{L44a}
Assume that the set $S$ obeys \eqref{Craig1}.
There exists $m>0$ and constants $K_{N}$ such that $\lim_{N\to \infty} K_{N} = 0$ and, for all $x \in \mathbb{R}$, and all $j\le N$,
\begin{equation}\label{expestimateKM0}
 \lVert \tilde\varphi^N(x,0) - \varphi(x,0) \rVert\le K_{N} e^{m\lvert x\rvert}.
\end{equation}
\end{lemma}

\begin{proof}[Proof of Prop.~\ref{P4a}]
By Lemma~\ref{L44a},
\[
\varphi(x,0) = \lim_{N\to \infty} \tilde\varphi^N(x,0)
\]
converges uniformly on compact sets of $x$. Denote
\[
W(x) = Q_1(\varphi(x,0)).
\]
Since $Q_1^N$ are uniformly continuous in $N$, we conclude that the sequence $W^N(x):=u_N(x,0)$ converges uniformly on compacts to $W(x)$.

We will now use a technique of Craig \cite[Section 5]{Cr}; in \cite{Cr} it is applied to a limit, uniform on compacts, of a sequence of translates of a given potential, but it applies equally well to the limit of the sequence $W^N$. Since the Schr\"odinger operators $H_{W^N}$ converge in strong resolvent sense to $H_W$, and since $\sigma(H_{W^N}) = S^N$, we conclude that $\sigma(H_W) \subset S$. Moreover, uniform convergence on compacts also implies convergence of the diagonal Green's function uniformly on compacts (in $z$) away from $S$. Since the $W^N$ are reflectionless, $\Re G(x,x;z+i0,W^N) = 0$ for a.e. $z\in S$, so by \cite[Lemma~5.2]{Cr}, $\Re G(x,x;z+i0,W) = 0$ for a.e.\ $z\in S$.

By analyticity, $\Im G(x,x;z,W) = 0$ for $z\in \mathbb{R}\setminus \sigma(H_W)$. Since $G(x,x;z+i0,W) = 0$ is possible only on a set of Lebesgue measure $0$, combining this with the previous conclusion shows that $S \setminus \sigma(H_W)$ has Lebesgue measure $0$. Since $S$ and $\sigma(H_W)$ are closed sets, \eqref{acproperty} then implies that $\sigma(H_W) = S$ and $W \in \mathcal{R}(S)$.

We will now prove that $B(W) = f$. Denote by $\mu_N(x)$, $\sigma^N(x)$ the Dirichlet data for $W^N$ (related by \eqref{mufromvarphi}, \eqref{sigmafromvarphi} to the angular Dirichlet data $\varphi^N(x)$), and by $\mu(x)$, $\sigma(x)$ their limits. Fix $j\in J$ and recall that, by the way $W^N$ were constructed, for $N\ge j$, $G(x,x;z,W^N) \ge 0$ for $z\in [\mu_j^N(x),E_j^+)$ and $G(x,x;z,W^N) \le 0$ for $z\in (E_j^-,\mu_j^N(x)]$. Using again convergence of the diagonal Green's function, we conclude that $G(x,x;z,W) \ge 0$ for $z\in [\mu_j(x),E_j^+)$ and $G(x,x;z,W) \le 0$ for $z\in (E_j^-,\mu_j(x)]$, so $\mu_j(x)$ are precisely the Dirichlet data of $W$. Since $\sigma_j(x)$ can be read off from whether $\mu_j(x)$ is increasing or decreasing, we conclude that the Dirichlet data of $W(x)$ are precisely $\varphi(x)$. In particular, $B(W) = \varphi(0) = f$.

If \eqref{traceconditions} also hold, then $Q_2^N, Q_3^N$ are uniformly continuous, so $Q_2^N\circ W^N$ and $Q_3^N \circ W^N$ are uniformly Cauchy on compacts.

Fix a compact $K\subset \mathbb{R}$. Since we already know that the sequence $W^N$ is Cauchy in $C(K)$, from \eqref{Q1} we conclude that $\partial_{x}^2 W^N$ is also Cauchy in $C(K)$, so $W^N$ is Cauchy in $C^2(K)$. Now, similarly, from \eqref{Q2} we conclude that $\partial_{x}^4 W^N$ is Cauchy in $C(K)$, so $W^N$ is Cauchy in $C^4(K)$. Its limit is $W$, so we conclude that $W \in C(K)$ and, by taking limits of \eqref{Q1}, \eqref{Q2}, that  \eqref{Q1W}, \eqref{Q2W} hold.

Finally, by using boundedness of $Q_1, Q_2, Q_3$, we conclude from \eqref{Q0W}, \eqref{Q1W}, \eqref{Q2W} that $W \in W^{4,\infty}(\mathbb{R})$.
\end{proof}

\begin{proof}[Proof of Prop.~\ref{P4}]
By Lemma~\ref{L44b}, convergence
\begin{equation}\label{varphixtlimit}
\varphi(x,t) = \lim_{N\to \infty} \tilde\varphi^N(x,t)
\end{equation}
is uniform on compacts in $(x,t) \in \mathbb{R}^2$. Introduce
\begin{equation}\label{uxt3}
u(x,t) = Q_1(\varphi(x,t)).
\end{equation}
By Prop.~\ref{P4a}, \eqref{uxt3} implies that $u(\cdot, t) \in \mathcal{R}(S)$ for each $t$ and $B(u(\cdot, t)) = \varphi(0,t)$.

In Prop.~\ref{P4a}, it was proved that for each $t\in \mathbb{R}$ and each compact $K \subset \mathbb{R}$,
\begin{equation}\label{uNtou}
\lim_{N\to\infty} \lVert u_N(\cdot, t) - u(\cdot, t)\rVert_{C^4(K)} = 0
\end{equation}
In the setting of this proposition, we know that convergence \eqref{varphixtlimit} is uniform on compacts in $\mathbb{R}^2$, so convergence \eqref{uNtou} is uniform on compacts in $t$. In particular, this implies that $\partial_x^k u$ is jointly continuous in $(x,t)$ for $k=0,1,2,3,4$.

Since $u_N$ obey the KdV equation \eqref{KdV},
\[
u_N(x,t) = u_N(x,0) + \int_0^t (6 u_N \partial_x u_N - \partial_x^3 u_N)(x,\tau) d\tau.
\]
By uniform convergence, we can take the limit $N\to\infty$ to conclude
\[
u(x,t) = u(x,0) + \int_0^t (6 u \partial_x u - \partial_x^3 u)(x,\tau) d\tau.
\]
Since the integrand is continuous, this implies that $u(x,t)$ is differentiable in $t$ and obeys \eqref{KdV}.
\end{proof}

\section{Linearization of the KdV Flow}\label{sec.5}

In this section, we will resume our investigation of the finite zone solutions $u_N(x,t)$ of the KdV equation and their limit $u(x,t)$. We will consider how these behave with respect to the generalized Abel map introduced by Sodin--Yuditskii. We assume throughout this section that $S$ is a regular Parreau--Widom set with finite gap length.

In order to construct a reparametrization of the isospectral torus, Sodin--Yuditskii~\cite{SY} introduced functions $\xi_j(z)$, one for each open gap $(E_j^-, E_j^+)$. The map $\xi_j(z)$  is the solution of the Dirichlet problem on $\mathbb{C} \setminus S$ with boundary conditions on $\bar S$ given by
\begin{equation}\label{bcomegam}
\xi_j(x) = \begin{cases}
1 & x = \infty\text{ or }x \in S, x \ge E_j^+ \\
0 & x \in S, x \le E_j^-
\end{cases}
\end{equation}
Regularity of $\bar S$ implies that $\xi_j(z)$ is a continuous function on $\overline{\mathbb{C}}$ and a harmonic function on $\mathbb{C}\setminus S$, and its values on $\bar S$ are given by \eqref{bcomegam}.

Let $\pi(\mathbb{C} \setminus S)$ be the fundamental group of $\mathbb{C} \setminus S$. It is a free group with the set of generators given by $\{c_j\}_{j\in J}$, where $c_j$ is a counterclockwise simple loop which intersects $\mathbb{R}$ at the points $\underline E - 1$ and $(E_j^- + E_j^+)/2$.

Following \cite{SY}, consider the group $\pi^*(\mathbb{C} \setminus S)$ of unimodular characters of $\pi(\mathbb{C} \setminus S)$. Use additive notation for $\pi^*(\mathbb{C} \setminus S)$. An element $\alpha \in \pi^*(\mathbb{C} \setminus S)$ is uniquely determined by its action on loops $c_j$, so we can write $\alpha = \{ \alpha_j \}_{j\in J}$ where $\alpha_j = \alpha(c_j) \in \mathbb{T}$. Endow $\pi^*(\mathbb{C}\setminus S)$ with the topology dual to the discrete topology of $\pi(\mathbb{C} \setminus S)$; there are many ways to choose a metric which induces this topology, for instance
\[
d(\alpha,\tilde \alpha) = \sum_{j\in J} \min( \lvert \alpha_j - \tilde \alpha_j \rvert, \gamma_j), \qquad \alpha, \tilde\alpha \in \pi^*(\mathbb{C} \setminus S),
\]
where $\gamma_j = E_j^+ - E_j^-$, as before.

\cite{SY} define the Abel map $A:\mathcal{D}(S) \to \pi^*(\mathbb{C} \setminus S)$ by defining its components $A_j = A(c_j)$,
\begin{equation}\label{linearizedvar}
A_j(\varphi) = \pi \sum_{k \in J} \sigma_k \; (\xi_j(\mu_k) - \xi_j(E_k^-)) \pmod{2 \pi\mathbb{Z}}
\end{equation}
where, as always, we assume that $\mu_k$, $\sigma_k$ are given in terms of $\varphi_k$ by \eqref{mufromvarphi}, \eqref{sigmafromvarphi}.

\cite{SY} prove the following (assuming that $S$ is a regular Parreau--Widom set with finite gap length):
\begin{enumerate}[(i)]
\item For each $j$, the sum in \eqref{linearizedvar} converges absolutely and uniformly in $\varphi \in \mathcal{D}(S)$; in particular, the map $A$ is well-defined;
\item $A$ is a homeomorphism between $\mathcal{D}(S)$ and $\pi^*(\mathbb{C} \setminus S)$;
\item  Denote by $A^N: \mathcal{D}(S^{N}) \to \pi^*(\mathbb{C} \setminus S^N)$ the Abel map for $S^N$. Project $\mathcal{D}(S)$ to $\mathcal{D}(S^{N})$ by truncation, and embed $\pi^*(\mathbb{C} \setminus S^N)$ into $\pi^*(\mathbb{C} \setminus S)$ by assuming $\alpha^N(c_j) = 0$ for $j> N$. With those conventions, consider $A^N$ as a map from $\mathcal{D}(S)$ to $\pi^*(\mathbb{C} \setminus S)$. Then, $A^{N} \to A$ as $N\to \infty$, uniformly on $\mathcal{D}(S)$.
\end{enumerate}

\cite{SY} introduced this map to generalize the notion of Jacobi inversion, which exists in the finite-gap setting. In that setting, it is known (see, e.g., \cite[Thm.~1.44]{GH}) that Jacobi inversion linearizes translation and KdV flows. Therefore
\begin{equation}\label{ANlinear}
A^{N}(\varphi^{N}(x,t)) = A^N(\varphi^N(0,0)) + \delta^N x + \zeta^N t
\end{equation}
for some $\delta^N, \zeta^N \in \mathbb{R}^N$.

We can now define the map
\[
\mathcal{M} = B^{-1} \circ A^{-1}: \pi^*(\mathbb{C} \setminus S) \to \mathcal{R}(S).
\]

\begin{prop}\label{P5}
Assume that $S$ is a regular Parreau--Widom set which obeys \eqref{Craig2}, \eqref{traceconditions}, \eqref{acproperty}. Then, the map $\mathcal{M}$ is a homeomorphism if $\mathcal{R}(S)$ is equipped with the metric inherited from $W^{4,\infty}(\mathbb{R})$.
\end{prop}

Recall that by Prop.~\ref{P4a}, we already know that $\mathcal{R}(S) \subset W^{4,\infty}(\mathbb{R})$.

\begin{proof}
For any $\epsilon_1 >0$, there exist $\epsilon_2, \epsilon_3 > 0$ such that
\begin{align*}
\lVert \alpha - \tilde \alpha \rVert_{\pi^*(\mathbb{C}\setminus S)} < \epsilon_3 & \implies \sup_{x\in\mathbb{R}} \lVert (\alpha+\delta x) - (\tilde \alpha+\delta x) \rVert_{\pi^*(\mathbb{C}\setminus S)}  < \epsilon_3 \\
 & \implies \sup_{x\in\mathbb{R}} \lVert A^{-1}(\alpha+\delta x) - A^{-1}(\tilde \alpha+\delta x) \rVert_{\mathcal{D}(S)}  < \epsilon_2 \\
 & \implies \sup_{k\in\{1,2,3\}} \sup_{x\in\mathbb{R}} \lvert Q_k(A^{-1}(\alpha+\delta x)) - Q_k(A^{-1}(\tilde \alpha+\delta x)) \rvert  < \epsilon_1
\end{align*}
Denote $W=\mathcal{M}(\alpha)$, $\tilde W = \mathcal{M}(\tilde\alpha)$. By \eqref{Q0W},
the above estimate for $k=1$ implies
\[
 \lVert W - \tilde W \rVert_\infty < \epsilon_1,
\]
so $\mathcal{M}$ is continuous as a map to $L^\infty(\mathbb{R})$. Inductively, using \eqref{Q1W} and \eqref{Q2W}, we conclude that $\mathcal{M}$ is a continuous as a map into $W^{2,\infty}(\mathbb{R})$ and $W^{4,\infty}(\mathbb{R})$, respectively.
\end{proof}

\begin{proof}[Proof of Theorem~\ref{Tap}]
Since $H_V$ is almost periodic and $\sigma(H_V) = \sigma_\ac(H_V) = S$, a result of Remling~\cite{R07} implies $V \in \mathcal{R}(S)$. Therefore, $V$ corresponds to a set of Dirichlet data $f = B(V)$, as defined in \eqref{mapB}.

By results of Rybkin~\cite{Ry} reviewed in Section 3, if $u(x,t)$ is a solution of \eqref{KdV}, \eqref{KdVinitial} which obeys \eqref{boundedness}, then $u(\cdot, t)\in\mathcal{R}(S)$ for each $t\in\mathbb{R}$, and the Dirichlet data $\varphi(x,t)$ of $u$ obey
\begin{equation}\label{varphixtflows}
\partial_x \varphi = \Psi(\varphi), \qquad \partial_t \varphi = \Xi(\varphi)
\end{equation}
and $\varphi(0,0)=f$. This determines $\varphi(x,t)$ uniquely, since $\Psi, \Xi$ are Lipshitz vector fields. The trace formula $u=Q_1\circ \varphi$ then determines $u$ uniquely.

Conversely, by Prop.~\ref{P4}, the solution $\varphi(x,t)$ of \eqref{varphixtflows} with $\varphi(0,0)=f$ exists and generates a solution of \eqref{KdV} by $u=Q_1 \circ \varphi$. Since $V, u(\cdot,0) \in\mathcal{R}(S)$ and $B(V) = B(u(\cdot,0))=f$, we conclude that $V = u(\cdot, 0)$, so $u$ obeys the correct initial condition \eqref{KdVinitial}. We have therefore established existence and uniqueness.

Since $H_V$ has purely a.c.\ spectrum, the set $S$ obeys \eqref{acproperty}, so Prop.~\ref{P5} applies to $S$.

We now recall the functions $\varphi^N(x,t)$ introduced in Section 4. Since $\varphi^N \to \varphi$ uniformly on compacts and $A^N \to A$ uniformly, we can conclude that $A^N(\varphi^N(x,t))$ converge uniformly on compacts to $A(\varphi(x,t))$. Taking the $j$-th component of \eqref{ANlinear}, it follows from uniform convergence that the limits
\[
\delta_j = \lim_{N\to\infty} \delta_j^N, \qquad \zeta_j = \lim_{N\to\infty} \zeta_j^N
\]
exist and
\[
A_j(\varphi(x,t)) = A_j(\varphi(0,0)) + \delta_j x + \zeta_j t.
\]
In particular, $\mathcal{M}^{-1}(u(\cdot,t)) = \mathcal{M}^{-1}(V) + \zeta t$.
\end{proof}

\section{Application to Small Quasi-Periodic Initial Data}\label{sec.6}

In this section, we will show that Theorem~\ref{TQ} follows from Theorem~\ref{Tap}. This will rely on spectral properties of Schr\"odinger operators with small quasiperiodic initial data, which have been extensively studied in \cite{El,DG1,DGL1,DGL2,DGL3}.

Until now, we have indexed gaps by the abstract gap label $J$, and in some sections it was notationally convenient to assume $J$ to be equal to $\mathbb{N}$. However, for almost periodic Schr\"odinger operators, there is a natural gap label given by the rotation number, see \cite{JM}. In the setting of Theorem~\ref{TQ} this means that gaps are naturally labelled by $m\in \zv$, with $m$ and $-m$ corresponding to the same gap and with $m=0$ corresponding to the bottom of the spectrum. We will use this label from now on. We remind the reader that in this paper, for $m=(m_1,\dots, m_\nu) \in \zv$, $\lvert m\rvert$ stands for the $\ell^1$ norm, $\lvert m\rvert = \sum_{j=1}^\nu \lvert m_j\rvert$.

In the setting of Theorem~\ref{TQ}, it is known that there is an $\epsilon_0(a_0,b_0,\kappa_0)>0$ such that, if $\epsilon < \epsilon_0$, $\omega$ obeys the Diophantine condition \eqref{eq:dioph}, and $V \in\mathcal{P}(\omega,\epsilon,\kappa_0)$, then, with $S=\sigma(H_V)$:
\begin{enumerate}[(i)]
\item $H_V$ has purely a.c.\ spectrum;
\item For every $m\in \zv \setminus\{0\}$,
\begin{equation}\label{gammam}
\gamma_m < 2 \epsilon \exp\left( -\frac{\kappa_0}2 \lvert m\rvert \right)
\end{equation}
\item For every $m \in \zv\setminus\{0\}$ and $n \in \zv$ with $m \neq n$ and $\lvert m\rvert \ge \lvert n\rvert$, we have
\begin{equation}\label{etamnlower}
\eta_{m,n} \ge a \lvert m \rvert^{-b},
\end{equation}
for some constants $a,b >0$ depending only on $\omega, \epsilon,\kappa_0, \nu$.
\item For every $m\in \zv\setminus\{0\}$,
\begin{equation}\label{etam0upper}
\eta_{m,0} \le c \lvert m\rvert^{2},
\end{equation}
for some constant $c>0$ depending only on $\omega, \epsilon$.
\item The set $S$ is homogeneous in the sense of Carleson: \eqref{Carleson} holds for $\tau=1/2$.

\item If $W \in \mathcal{R}(S)$, then $W$ is also quasiperiodic with the same frequency $\omega$ and similar decay properties of Fourier coefficients; more precisely, $W \in \mathcal{P}(\omega,\sqrt{4\epsilon},\kappa_0/4)$.
\end{enumerate}

We will need a subexponential estimate of the constants
\[
C_m = (\eta_{m,0} + \gamma_m)^{1/2} \prod_{\substack{n \in J \\ n\neq m}} \left( 1 + \frac{\gamma_n}{\eta_{m,n}} \right)^{1/2}.
\]
Fix a constant $L = L(\epsilon,\kappa_0,a,b)$ such that
\[
2 \epsilon \exp\left( -\frac{\kappa_0}2 \lvert m\rvert \right) < L a^4 \lvert m \rvert^{-4b}, \quad \forall m\in \zv\setminus\{0\}.
\]
Define
\[
R_m = \left\{ n \in \zv\setminus\{0\} \mid n\neq m, \gamma_n > L \eta_{n,m}^4 \right\}
\]
It follows from \eqref{gammam}, \eqref{etamnlower}, and our choice of $L$, that $n\in P_m$ implies $\lvert n \rvert \le \lvert m\rvert$. In particular, the set $R_m$ is finite. We order its elements by decreasing $\ell^1$ norm, i.e., we denote $R_m = \{ n_m^{(r)} \}_{r=0}^{r_0}$ with
\begin{equation}\label{nmordered}
\lvert n_m^{(0)} \rvert \ge \dots \ge \lvert n_m^{(r_0)} \rvert.
\end{equation}

\begin{lemma}\begin{enumerate}[(i)]
\item There are constants $\tau,\beta$ which depend only on $a, b, \kappa_0, \nu$, such that
\begin{equation}\label{nmrjump}
\lvert n_m^{(r)} \rvert \ge \tau \exp\left( \beta \lvert n_m^{(r+2)} \rvert \right), \quad \forall  r \le r_0 - 2.
\end{equation}
\item There is a constant $D = D(a,b,\kappa_0,\nu)$ such that
\[
\lvert R_m \rvert \le \log_2 \log_2 \lvert m\rvert + D.
\]
\item There is a constant $F = F(a,b,\kappa_0,\nu,\omega)$ such that
\begin{equation}\label{Cmsubexp}
C_m \le F \exp( F \log \lvert m\rvert \log \log \lvert m\rvert).
\end{equation}
\end{enumerate}
\end{lemma}

\begin{proof}
(i) Among $n_m^{(r)}, n_m^{(r+1)}, n_m^{(r+2)}$, we can find two for which the corresponding gaps lie on the same side of the $m$-th gap; denote them $s = n_m^{(r_1)}, t= n_m^{(r_2)}$, with $r_1<r_2$. By the ordering \eqref{nmordered}, since $r \le r_1 < r_2 \le r+2$, it suffices to prove
\begin{equation}\label{stAB}
\lvert s \rvert \ge \tau \exp\left( \beta \lvert t \rvert \right).
\end{equation}
Since the gaps labelled by $s$, $t$ are on the same side of the $m$-th gap,
\[
\eta_{s, t} \le \max (\eta_{m, t}, \eta_{m,s} ) <  \max( L^{-1/4} \gamma_{t}^{1/4}, L^{-1/4} \gamma_{s}^{1/4} ) \le L^{-1/4} (2 \epsilon)^{1/4} \exp\left( -\frac{\kappa_0}8 \lvert t \rvert \right)
\]
Combining this with \eqref{etamnlower} gives
\[
a \lvert s \rvert^{-b} < L^{-1/4} (2 \epsilon)^{1/4} \exp\left( -\frac{\kappa_0}8 \lvert t\rvert \right),
\]
which can be rewritten in the form \eqref{stAB}.

(ii) Pick $R=R(a,b,\kappa_0,\nu)\ge 2$ such that $\tau \exp(\beta x) \ge x^4$ for $x \ge R$. Assume that $\lvert n_m^{(r)} \rvert > R$ for all $r \le  2 r_1$. Then \eqref{nmrjump} gives
\[
\lvert m\rvert \ge \lvert n_m^{(0)} \rvert \ge \lvert n_m^{(2)} \rvert^4 \ge \dots \ge \lvert n_m^{(2r_1)} \rvert^{4^{r_1}} \ge 2^{4^{r_1}},
\]
which implies $2 r_1 \le \log_2 \log_2 \lvert m\rvert$, bounding the number of $n_m^{(r)}$ with $\lvert n_m^{(r)}\rvert > R$. Trivially, the number of $n_m^{(r)}$ with $\lvert n_m^{(r)} \rvert \le R$ is less than $(2R+1)^\nu$. Combining the two estimates completes the proof.

(iii) To estimate $C_m$, we estimate first the part of the product with $n\notin R_m$,
\[
\prod_{\substack{n \in J \setminus R_m}} \left( 1 + \frac{\gamma_n}{\eta_{m,n}} \right)^{1/2} \le \prod_{\substack{n \in J \setminus R_m}} \left( 1 + L^{1/4} \gamma_n^{1/4} \right)^{1/2} \le \exp\left( \frac 12 \sum_{n\in J} L^{1/4} \gamma_n^{1/4} \right) \le \tilde C
\]
by \eqref{gammam}. For the part of the product with $n\in R_m$, use $\lvert n\rvert \le \lvert m\rvert$ and the bound on $\lvert R_m\rvert$ to conclude that
\[
\prod_{\substack{n \in R_m}} \left( 1 + \frac{\gamma_n}{\eta_{m,n}} \right)^{1/2} \le \prod_{\substack{n \in R_m}} \left( 1 + \frac{2 \epsilon}{a \lvert m\rvert^{-b}} \right)^{1/2} \le \left( 1 + 2\epsilon a^{-1} \lvert m\rvert^b \right)^{\log_2 \log_2 \lvert m\rvert + D}.
\]
Finally, the factor $\eta_{m,0}+\gamma_m$ is polynomially bounded by \eqref{etam0upper} and \eqref{gammam}.
\end{proof}

\begin{proof}[Proof of Theorem~\ref{TQ}]
In order to apply Theorem~\ref{Tap}, it only remains to verify that $H_V$ obeys the conditions \eqref{Craig2}, \eqref{traceconditions}.

\eqref{traceconditions} follows immediately from
\[
\sum_{m\in\zv\setminus\{0\}} (1+\eta_{m,0}^2) \gamma_m \le  2\epsilon (1+a^2) \sum_{m\in\zv\setminus\{0\}}  \lvert m\rvert^{2b} \exp\left( -\frac{\kappa_0}2 \lvert m\rvert \right) < \infty.
\]
Similarly, $\sum \gamma_m^{1/2} < \infty$ follows immediately from \eqref{gammam}. The second inequality in \eqref{Craig2},
\[
\sup_{m\in J} \gamma_m^{1/2} \frac{1+\eta_{m,0}}{\eta_{m,0}} C_m < \infty,
\]
holds since the exponential decay of $\gamma_m$ controls the (at most subexponential) growth of $\eta_{m,0}^{-1}$ and of $C_m$. For the third inequality in \eqref{Craig2}, use
\[
\frac{\gamma_m^{1/2} \gamma_n^{1/2}}{\eta_{m,n}} \le (2\epsilon L)^{1/4} \gamma_m^{1/4} \gamma_n^{1/4}
\]
for all $m,n\in J$ to conclude that
\[
\sup_{m\in J} \sum_{\substack{n \in J  \\ n \neq m}} \left( \frac{\gamma_m^{1/2} \gamma_n^{1/2}}{\eta_{m,n}} \right)^a (1 + \eta_{m,0} ) C_m  \le \sup_{m\in J} \sum_{\substack{n \in J  \\ n \neq m}} \left( \gamma_m^{1/4} \gamma_n^{1/4} \right)^a (1 + \eta_{m,0} ) C_m < \infty,
\]
again by the observation that the exponential decay of $\gamma_m$ controls the at most subexponential growth of the other factors.

Therefore, Theorem~\ref{Tap} applies to the initial data $V \in \mathcal{P}(\omega,\epsilon,\kappa_0)$. This implies Theorem~\ref{TQ}, since $\mathcal{R}(S) \subset \mathcal{P}(\omega,\sqrt{4\epsilon},\kappa_0/4)$ by \cite{DGL2}.
\end{proof}

\end{document}